\documentclass[11pt]{article}

\usepackage{amsmath,amssymb,amsthm,mathtools}
\usepackage{geometry}
\usepackage{enumitem}
\usepackage{booktabs}
\usepackage{graphicx}
\usepackage{tikz}
\usepackage{float}
\usepackage{diagbox}
\usepackage[T1]{fontenc}
\usepackage{lmodern}
\usepackage{microtype}
\usepackage{hyperref}
\usepackage{tabularx}
\hypersetup{colorlinks=true, linkcolor=blue, citecolor=blue, urlcolor=blue,
  bookmarksdepth=3}

\newtheorem{theorem}{Theorem}
\newtheorem{lemma}{Lemma}
\newtheorem{corollary}{Corollary}
\newtheorem{proposition}{Proposition}
\newtheorem{definition}{Definition}
\newtheorem{remark}{Remark}
\newtheorem{example}{Example}

\newcommand{\Zm}{\mathbb Z_m}

\newcommand{\bulk}{\mathrm{bulk}}   % generic (non-stall) step; see \S\ref{app:routeE-cycle} for definition

\title{Hamilton decompositions of the directed 3-torus: a return-map and odometer view}
\author{Sanghyun Park, \\ Yonsei University}
\date{March 2026}

\begin{document}
\maketitle

\begin{abstract}
We prove that the directed three-dimensional torus
$D_3(m)=\mathrm{Cay}((\mathbb Z_m)^3,\{e_1,e_2,e_3\})\cong \vec C_m\square \vec C_m\square \vec C_m$,
equivalently the Cartesian product of three directed $m$-cycles,
admits a decomposition into three arc-disjoint directed Hamilton cycles for every integer $m\ge 3$.

The proof is organized around return maps rather than the full $m^3$-vertex dynamics.
Since each step increases the layer function $S=i+j+k\pmod{m}$ by $1$, Hamiltonicity reduces to the $m$-step return maps on the layer-zero plane $\{S=0\}$ of size $m^2$.
These return maps are modeled on a two-dimensional odometer, with one coordinate acting as a clock and the other recording the carry.

For odd $m$, five Kempe swaps of the canonical coloring produce return maps explicitly affine-conjugate to the odometer.
For even $m$, a sign-product invariant shows that no Kempe-from-canonical construction can work, so we build a different explicit low-layer assignment whose return maps are finite-defect perturbations of the same clock-and-carry mechanism.
A further first-return reduction converts the remaining closure into a finite splice analysis; the boundary case $m=4$ is handled by a direct finite witness.
The construction has been formalized in Lean~4.
\end{abstract}

\section{Introduction}

Fix an integer $m\ge 3$ and write $\mathbb Z_m=\mathbb Z/m\mathbb Z$.
Let $V=(\mathbb Z_m)^3$ and let $D_3(m)$ be the directed Cayley graph with arc set
\[
(i,j,k)\to(i+1,j,k),\qquad (i,j,k)\to(i,j+1,k),\qquad (i,j,k)\to(i,j,k+1),
\]
with all arithmetic modulo $m$.
Equivalently,
\[
D_3(m)=\mathrm{Cay}((\mathbb Z_m)^3,\{e_1,e_2,e_3\})\cong \vec C_m\square \vec C_m\square \vec C_m.
\]
A \emph{directed Hamilton decomposition} of $D_3(m)$ is a partition of its $3m^3$ arcs into three arc-disjoint directed Hamilton cycles.

Hamilton decompositions of products of cycles are classical on the undirected side.
Kotzig proved that $C_m\square C_n$ decomposes into Hamilton cycles~\cite{Kotzig1973}, and Foregger extended this to products of several cycles~\cite{Foregger1978}; see also the surveys of Bermond~\cite{Bermond1978} and Alspach, Bermond, and Sotteau~\cite{AlspachBermondSotteau1990}, together with Stong's work on graph products~\cite{Stong1991}.
For directed products, Trotter and Erd\H{o}s characterized when the Cartesian product of two directed cycles is Hamiltonian~\cite{TrotterErdos1978}, and Curran and Witte showed that the product of any three or more nontrivial directed cycles has a directed Hamilton cycle~\cite{CurranWitte1985}.
On the decomposition side, Keating determined when the product of two directed cycles decomposes into two directed Hamilton cycles~\cite{Keating1985}, while Bogdanowicz studied decompositions into directed cycles of prescribed common length and isolated Hamilton cycles under explicit arithmetic hypotheses~\cite{Bogdanowicz2017,Bogdanowicz2020}.
Darijani, Miraftab, and Morris resolved the corresponding problem for arc-disjoint Hamilton \emph{paths} in most cases and identified the three-factor product as the remaining boundary in that setting~\cite{DarijaniMiraftabMorris2022}.
As an abelian Cayley digraph, $D_3(m)$ also lies in the broader Hamiltonicity tradition surveyed by Witte--Gallian~\cite{WitteGallian1984} and Curran--Gallian~\cite{CurranGallian1996}.
The symmetric three-generator torus $D_3(m)$ therefore sits at a natural boundary between existing Hamiltonicity theorems, two-factor decomposition results, and the remaining three-factor decomposition problem.
\paragraph{Contemporaneous and independent work.}
During the preparation of this manuscript,
Knuth's note \emph{Claude's Cycles}~\cite{Knuth2026} gave a different construction for odd $m$; the revised March 2026 version also records independent solutions of the even case.
Independently, Aquino-Michaels~\cite{AquinoMichaels2026} developed another program for the same problem, giving different odd- and even-case constructions together with computational exploration.
The present paper gives a self-contained proof for all $m\ge 3$, based on a Kempe sign-product invariant and a unified return-map/odometer framework, and includes a Lean~4 formalization.

The proof is organized around a return-map viewpoint.
Define the layer function
\[
S(i,j,k)=i+j+k\pmod m.
\]
Every legal arc increases $S$ by $1$.
Thus, for any color, an orbit returns to the plane
\[
P_0=\{S=0\}
\]
exactly once every $m$ steps.
Hamiltonicity on the full vertex set $V$ therefore reduces to the cycle structure of an $m$-step return map on $P_0$, a set of size $m^2$.
The model dynamics on that section are the standard odometer
\[
O(u,v)=(u+1,\ v+\mathbf 1_{u=0}),
\]
whose first coordinate is a clock and whose second coordinate records the carry when the clock wraps.
The odd and even arguments are best understood as two different ways of realizing this same clock-and-carry mechanism.

The proof contributes three distinct ingredients.
\begin{enumerate}[label=\textup{(\roman*)},leftmargin=2.5em]
\item a \emph{sign-product invariant} for Kempe swaps, which yields a parity obstruction showing that even $m$ cannot be reached from the canonical coloring by Kempe swaps alone;
\item an explicit \emph{odd-case five-swap witness} whose three return maps on $P_0$ are all affine-conjugate to the standard odometer;
\item an explicit \emph{even-case Route~E construction} on the three lowest layers, whose return maps become finite-defect odometer dynamics after adapted coordinate changes and one further first-return reduction to lane transversals.
\end{enumerate}

The odd case is therefore an exact affine odometer theorem.
The even case is a repair theorem: after the parity obstruction blocks the canonical Kempe route, we rebuild a primitive clock and carry by modifying only the low layers.
The boundary case $m=4$ is treated separately by a finite witness.

Our main result is the following.

\begin{theorem}\label{thm:main-all}
For every integer $m\ge 3$, the arc set of $D_3(m)$ decomposes into three directed Hamilton cycles.
\end{theorem}

In proof terms, the paper has a simple architecture.
Section~\ref{sec:prelim} develops the two global reductions: the sign-product invariant for Kempe swaps and the return-map reduction from $V$ to $P_0$.
Section~\ref{sec:affine} proves the odd case by writing down a five-swap coloring and then exhibiting explicit affine conjugacies from its return maps to the odometer.
Section~\ref{sec:even} begins the even case by isolating the clock-and-carry mechanism that must be rebuilt, introduces Route~E, and states the finite-defect return-map framework.
Appendices~\ref{app:routeE}--\ref{app:routeE-cycle} are part of the written proof: they derive the closed formulas for the Route~E return maps and then turn those formulas into single-cycle statements.
Appendix~\ref{app:m4} records the finite witness for $m=4$.
Appendix~\ref{app:verification} is supplementary only; it audits the displayed formulas and witness data but is not used as a proof step.

The point of this organization is conceptual as well as technical.
The layer function supplies a built-in clock of period $m$.
Passing to the first return removes that forced drift and exposes the genuine slow variables.
In the odd case nothing further is needed: the return maps on $P_0$ are already affine copies of the standard odometer.
In the even case one more first-return reduction to explicit lanes isolates the same mechanism, now with finitely many controlled splices.
That is the sense in which the whole proof is a clock-and-carry argument.

\begin{table}[t]
\centering
\footnotesize
\begin{tabularx}{\textwidth}{@{}l>{\raggedright\arraybackslash}X@{}}
\toprule
notation & meaning \\
\midrule
$V=(\mathbb Z_m)^3$ & vertex set of the directed torus $D_3(m)$ \\
$S(i,j,k)=i+j+k\pmod m$ & layer function; every legal arc increases $S$ by $1$ \\
$P_t=\{S=t\}$ & layer plane / return section \\
$f_c$ & global color-$c$ map on $V$ (a permutation when the direction assignment is a valid coloring) \\
$F_c=f_c^m|_{P_0}$ & return map on $P_0$ for the odd-case affine construction (Section~\ref{sec:affine}) \\
$R_c=f_c^m|_{P_0}$ & return map on $P_0$ for Route~E (Section~\ref{sec:even}); abbreviated $R_c$ to distinguish from the odd-case $F_c$ \\
$O(u,v)=(u+1,\ v+\mathbf 1_{u=0})$ & standard two-dimensional odometer; model return map after the layer drift is factored out \\
$T_c$ & first-return map on the lane transversal $L=\{(u,0):u\in\mathbb Z_m\}$; this is the one-dimensional carry map in the even proof \\
$\rho_c(u)$ & first-return time from lane $u$ \\
$(u,t)=\Phi_c(i,k)$ & adapted bulk coordinates for color~$c$; $t$ is the clock coordinate in the bulk branch (Lemma~\ref{lem:routeE-bulkcoords}) \\
\bottomrule
\end{tabularx}
\caption{Core notation and proof-level roles in the return-map analysis of $D_3(m)$.}
\label{tab:notation}
\end{table}

\section{Colorings, Kempe swaps, and return maps}\label{sec:prelim}

\subsection{Arc-colorings and direction assignments}
We define a \emph{coloring} of $D_3(m)$ to be a triple of permutations $(f_0,f_1,f_2)$ on $V$ such that
\[
\{f_0(v),f_1(v),f_2(v)\}=\{v+e_1,\ v+e_2,\ v+e_3\}
\qquad\text{for every }v\in V.
\]
Equivalently, at each vertex this assigns a bijection from the color set $\{0,1,2\}$ to the three out-arcs $\{v\to v+e_1,\, v\to v+e_2,\, v\to v+e_3\}$.
The color-$c$ subgraph is the functional digraph of $f_c$ (i.e., the digraph with arc set $\{v\to f_c(v):v\in V\}$) and is therefore a disjoint union of directed cycles.
We say that color $c$ is \emph{Hamilton} if $f_c$ is a single $m^3$-cycle.

\begin{example}[Canonical coloring]\label{ex:canonical}
The \emph{canonical coloring} assigns the identity direction triple
$(d_0,d_1,d_2)=(0,1,2)$ at every vertex, so that
\[
f_0(i,j,k)=(i+1,j,k),\qquad f_1(i,j,k)=(i,j+1,k),\qquad f_2(i,j,k)=(i,j,k+1).
\]
Each color is then a disjoint union of $m^2$ directed $m$-cycles.
This coloring serves as the starting point for the odd-case construction (Section~\ref{sec:affine})
and as the baseline whose sign product determines the parity barrier (Corollary~\ref{cor:parity-barrier}).
\end{example}

We will also need a relaxed notion that drops the permutation requirement.

\begin{definition}[Direction assignment]\label{def:direction-assignment}
A \emph{direction assignment} on $D_3(m)$ is a triple $\delta=(d_0,d_1,d_2)$ with $d_c:V\to\{0,1,2\}$ such that for every vertex $v$ the triple $(d_0(v),d_1(v),d_2(v))$ is a permutation of $(0,1,2)$.
It induces maps
\[
g_c(v):=v+e_{d_c(v)}\qquad (c\in\{0,1,2\}).
\]
We call the direction assignment a \emph{coloring} precisely when each induced map $g_c$ is a permutation of $V$.
\end{definition}

Route~E is first specified as a direction assignment; Lemma~\ref{lem:return-validity} later upgrades it to a coloring.

\subsection{Kempe maps and Kempe swaps}
Kempe-chain recoloring is a classical tool in graph coloring; for background on Kempe equivalence see Mohar~\cite{Mohar2006} and the monograph of Stiebitz, Scheide, Toft, and Favrholdt~\cite{StiebitzEtAl2012}.
Fix two colors $r\ne s$.
Given a coloring $(f_0,f_1,f_2)$, define the \emph{Kempe map}
\[
\tau_{r,s}:=f_s^{-1}\circ f_r.
\]
Its cycles are exactly the alternating $(r,s)$-colored cycles in the $2$-colored subdigraph on colors $r$ and $s$.

\begin{definition}[Kempe swap]
Let $X\subseteq V$ be a union of cycles of $\tau_{r,s}$.
The \emph{Kempe swap of colors $(r,s)$ on $X$} is the operation producing a new triple $(f'_0,f'_1,f'_2)$ by exchanging the color-$r$ and color-$s$ arcs at every vertex in $X$:
\[
f'_r(v)=
\begin{cases}
f_s(v), & v\in X,\\
f_r(v), & v\notin X,
\end{cases}
\qquad
f'_s(v)=
\begin{cases}
f_r(v), & v\in X,\\
f_s(v), & v\notin X,
\end{cases}
\qquad
f'_t=f_t\ (t\notin\{r,s\}).
\]
\end{definition}

\begin{lemma}[Kempe swaps preserve validity]\label{lem:kempe-valid}
If $(f_0,f_1,f_2)$ is a coloring and $X$ is a union of cycles of $\tau_{r,s}$, then the Kempe swap on $X$ produces another valid coloring.
\end{lemma}

\begin{proof}
On each $\tau_{r,s}$-cycle $C\subseteq X$, the color-$r$ and color-$s$ arcs restrict to two perfect matchings from the source set $C$ onto the common image set $f_r(C)=f_s(C)$.
Indeed, if $C=(v_0,v_1,\dots,v_{\ell-1})$, then $f_r(v_t)=f_s(v_{t+1})$ along the cycle.
Swapping the two colors on $C$ simply exchanges these matchings, so every vertex incident with $C$ still has indegree $1$ and outdegree $1$ in each of the two colors.
Outside $X$ nothing changes.
\end{proof}

\subsection{A parity barrier}
For a permutation $\pi$ on a finite set, let $\mathrm{sgn}(\pi)\in\{\pm1\}$ denote its sign.

\begin{theorem}[Sign-product invariant]\label{thm:sign-invariant}
Let $(f_0,f_1,f_2)$ be a coloring and let $(f'_0,f'_1,f'_2)$ be obtained by a Kempe swap between two colors.
Then
\[
\mathrm{sgn}(f'_0)\,\mathrm{sgn}(f'_1)\,\mathrm{sgn}(f'_2)
=
\mathrm{sgn}(f_0)\,\mathrm{sgn}(f_1)\,\mathrm{sgn}(f_2).
\]
\end{theorem}

\begin{proof}
It suffices to consider a swap between colors $r$ and $s$, since the third color is unchanged.
Because the swap support $X$ is a disjoint union of $\tau_{r,s}$-cycles and the sign is multiplicative, it suffices to treat a single cycle.
Let $C=(v_0,v_1,\dots,v_{\ell-1})$ be one cycle of $\tau_{r,s}=f_s^{-1}\circ f_r$ contained in $X$.
Define $\sigma\in S(V)$ to be the $\ell$-cycle $(v_0\,v_1\,\cdots\,v_{\ell-1})$ on $C$ and to fix $V\setminus C$ pointwise.
A direct check using $f_r(v_t)=f_s(v_{t+1})$ shows
\[
f'_r=f_r\circ\sigma^{-1},\qquad f'_s=f_s\circ\sigma.
\]
Therefore
\[
\mathrm{sgn}(f'_r)\mathrm{sgn}(f'_s)
=
\mathrm{sgn}(f_r)\mathrm{sgn}(\sigma^{-1})\,\mathrm{sgn}(f_s)\mathrm{sgn}(\sigma)
=
\mathrm{sgn}(f_r)\mathrm{sgn}(f_s).
\]
When $X$ is a union of several cycles, applying the argument cycle by cycle shows that the swap on each cycle preserves the sign product, so the global sign product is invariant.
\end{proof}

\begin{corollary}[Parity barrier for Kempe-from-canonical when $m$ is even]\label{cor:parity-barrier}
Let $m$ be even. Then no sequence of Kempe swaps starting from the canonical coloring can reach a Hamilton decomposition of $D_3(m)$.
\end{corollary}

\begin{proof}
A directed Hamilton cycle on $N=m^3$ vertices is an $N$-cycle permutation and therefore has sign $(-1)^{N-1}=-1$ when $m$ is even.
Hence any decomposition into three Hamilton cycles must have sign product $(-1)^3=-1$.
In the canonical coloring, each color is a disjoint union of $m^2$ directed $m$-cycles.
Each $m$-cycle has sign $(-1)^{m-1}=-1$ since $m$ is even, and the product of $m^2$ such signs is $(-1)^{m^2}=+1$ since $m^2$ is even.
Therefore the total sign product is $(+1)^3=+1$.
The result now follows from Theorem~\ref{thm:sign-invariant}.
\end{proof}

\subsection{The layer function and return maps}
Define the layer function
\[
S(i,j,k):=i+j+k\pmod m.
\]
Every allowed arc increases $S$ by $1$, so for every coloring and every color $c$ one has $S(f_c(v))=S(v)+1$.

\begin{lemma}[Plane invariance of Kempe maps]\label{lem:plane-invariant}
For any coloring and any two colors $r\ne s$, the Kempe map $\tau_{r,s}=f_s^{-1}\circ f_r$ preserves $S$.
Consequently, every $\tau_{r,s}$-cycle is contained in a single plane $P_t:=\{S=t\}$.
\end{lemma}

\begin{proof}
Let $w=f_r(v)$. Then $S(w)=S(v)+1$.
If $u=\tau_{r,s}(v)=f_s^{-1}(w)$, then $f_s(u)=w$, so $S(w)=S(u)+1$ and therefore $S(u)=S(v)$.
\end{proof}

\begin{corollary}[Plane swaps are always Kempe-valid]\label{cor:plane-swaps}
For any $t\in\mathbb Z_m$, the plane $P_t$ is a union of cycles of $\tau_{r,s}$ for every pair $r\ne s$.
So swapping two colors on an entire plane is always a valid Kempe swap.
\end{corollary}

\begin{lemma}[Return maps control cycle structure]\label{lem:return}
Fix a coloring and a color $c$.
Let $P_0=\{S=0\}$ and let
\[
F_c:=f_c^m\big|_{P_0}:P_0\to P_0
\]
be the return map.
Then the directed cycles of $f_c$ are in bijection with the directed cycles of $F_c$, and the length of a cycle of $f_c$ is $m$ times the length of the corresponding cycle of $F_c$.
In particular, $f_c$ is Hamilton on $V$ if and only if $F_c$ is a single $m^2$-cycle on $P_0$.
\end{lemma}

\begin{proof}
Along an $f_c$-orbit, the value of $S$ increases by $1$ at every step, so the orbit meets $P_0$ exactly once every $m$ steps.
The induced motion on $P_0$ is therefore exactly the return map $F_c=f_c^m|_{P_0}$.
\end{proof}

\begin{lemma}[Validity via return maps]\label{lem:return-validity}
Let $f:V\to V$ be a directed color map such that every step increases the layer $S$ by $1$.
Write $P_t=\{S=t\}$ and let $g_t:P_t\to P_{t+1}$ be the induced layer maps.
If the return map
\[
R=g_{m-1}\circ\cdots\circ g_0:P_0\to P_0
\]
is bijective, then each $g_t$ is bijective.
Consequently $f$ is a permutation of $V$.
\end{lemma}

\begin{proof}
All layers have the same finite size $m^2$; this equal-size property is essential for the argument.
Since $R$ is bijective, it is injective.
If $g_0$ were not injective, then $R$ would not be injective.
So $g_0$ is injective and therefore bijective.

Now write
\[
g_{m-1}\circ\cdots\circ g_1 = R\circ g_0^{-1},
\]
so the shorter composition is again bijective.
Repeating the same argument inductively shows that every $g_t$ is bijective.
Therefore every vertex has indegree $1$ and outdegree $1$ in this color, so $f$ is a permutation of $V$.
\end{proof}

\begin{remark}[Why the return section matters]\label{rem:return-odometer}
Lemma~\ref{lem:return} is the Poincar\'e-section reduction that drives the whole paper.
It shrinks the state space from $m^3$ vertices to a section of size $m^2$ and, more to the point, factors out the forced layer drift and isolates the slow carry information that decides whether a color is Hamiltonian.
The standard normal form for that reduced dynamics is the odometer introduced next.
\end{remark}

\begin{definition}[Standard two-dimensional odometer]
Define
\[
O:\Zm^2\to \Zm^2,\qquad O(u,v)=(u+1,\ v+\mathbf 1_{u=0}).
\]
\end{definition}

\begin{lemma}[The odometer is a single cycle]\label{lem:odometer}
The map $O$ is a single $m^2$-cycle on $\Zm^2$.
\end{lemma}

\begin{proof}
After $m$ applications of $O$, the first coordinate has completed one full turn and the carry has fired exactly once, so
\[
O^m(u,v)=(u,v+1).
\]
Hence $O^{m^2}=\mathrm{id}$.

Conversely, suppose $O^t(u,v)=(u,v)$.
Write $t=km+r$ with $0\le r<m$.
The first coordinate of $O^t(u,v)$ is $u+r$, so $r=0$.
Then
\[
O^t(u,v)=O^{km}(u,v)=(u,v+k),
\]
and therefore $k\equiv 0\pmod m$.
Thus every point has exact period $m^2$, and since $|\Zm^2|=m^2$, the map $O$ is a single cycle.
\end{proof}

\begin{remark}[Why the odometer matters]
The odometer is the archetypal clock-and-carry system.
One coordinate advances every step; the other changes only when that first coordinate wraps.
The odd proof will identify the return maps with this model exactly, up to affine conjugacy.
The even proof will recover the same mechanism only after a second return-map reduction to lane transversals and a finite splice analysis.
\end{remark}

\section{Odd \texorpdfstring{$m$}{m}: an explicit affine odometer witness}\label{sec:affine}

This section proves the odd case.
A five-swap Kempe modification supported on two adjacent layers produces a valid coloring with closed return-map formulas on $P_0$.
Those formulas are already affine odometers, so once they are written down, Hamiltonicity follows immediately from the odometer lemma.

\subsection{The affine supports}
Recall the canonical coloring from Example~\ref{ex:canonical}.

Let
\[
P_0:=\{S=0\},\qquad P_1:=\{S=1\},
\]
and define the affine lines
\[
L_0:=\{S=0,\ k=0\},\qquad L_1:=\{S=1,\ k=0\}.
\]

\begin{theorem}[Five-swap witness]\label{thm:5swap}
Start from the canonical coloring and perform the following five Kempe swaps:
\begin{enumerate}[leftmargin=2.2em]
\item swap colors $(0,1)$ on the line $L_0$;
\item swap colors $(0,2)$ on the plane $P_0$;
\item swap colors $(0,1)$ on the plane $P_0$;
\item swap colors $(0,1)$ on the line $L_1$;
\item swap colors $(0,2)$ on the plane $P_1$.
\end{enumerate}
This produces a valid coloring for every $m\ge 3$.
\end{theorem}

\begin{proof}[Proof of Theorem~\ref{thm:5swap}]
Steps (2), (3), and (5) are Kempe-valid by Corollary~\ref{cor:plane-swaps} because they are supported on entire planes.
For step (1), the coloring is canonical, and
\[
\tau_{0,1}=f_1^{-1}\circ f_0:(i,j,k)\mapsto(i+1,j-1,k)
\]
preserves both $S$ and $k$, acting on $L_0=\{S=0,k=0\}$ as a single $m$-cycle.
Hence (1) is Kempe-valid.

For step (4), note that steps (1)--(3) modify only tails in $P_0$.
If $v\in P_1$, then $f_0(v)$ lies in $P_2$.
Any color-$1$ arc landing in $P_2$ must originate in $P_1$, because every legal arc increases $S$ by $1$.
Hence $f_1^{-1}$ on the image of $P_1$ in $P_2$ depends only on the outgoing color-$1$ arcs from $P_1$, which remain canonical after steps (1)--(3).
Therefore $\tau_{0,1}$ acts on $P_1$ as $(i,j,k)\mapsto(i+1,j-1,k)$, so $L_1$ is a single $\tau_{0,1}$-cycle and step (4) is Kempe-valid.
\end{proof}

\begin{definition}[Bump maps]
For $v=(i,j,k)\in V$ define
\[
\mathrm{bump}_i(v)=(i+1,j,k),\qquad
\mathrm{bump}_j(v)=(i,j+1,k),\qquad
\mathrm{bump}_k(v)=(i,j,k+1),
\]
with all arithmetic modulo $m$.
\end{definition}

\begin{proposition}[Closed form of the odd witness]\label{prop:odd-closed}
Let $m\ge 3$ and let $v=(i,j,k)\in V$ with $S=i+j+k\pmod m$.
After the five swaps of Theorem~\ref{thm:5swap}, the resulting maps $(f_0,f_1,f_2)$ are
\[
f_0(v)=
\begin{cases}
\mathrm{bump}_j(v) & \text{if } S=0 \text{ and } k\neq 0,\\
\mathrm{bump}_k(v) & \text{if } S=1,\\
\mathrm{bump}_i(v) & \text{otherwise,}
\end{cases}
\]
\[
f_1(v)=
\begin{cases}
\mathrm{bump}_k(v) & \text{if } S=0,\\
\mathrm{bump}_i(v) & \text{if } S=1 \text{ and } k=0,\\
\mathrm{bump}_j(v) & \text{otherwise,}
\end{cases}
\]
\[
f_2(v)=
\begin{cases}
\mathrm{bump}_j(v) & \text{if } S\in\{0,1\} \text{ and } k=0,\\
\mathrm{bump}_i(v) & \text{if } S\in\{0,1\} \text{ and } k\neq 0,\\
\mathrm{bump}_k(v) & \text{otherwise.}
\end{cases}
\]
\end{proposition}

\begin{proof}
Outside $P_0\cup P_1$ none of the five swaps acts, so the coloring remains canonical there.

On $P_0$, step (1) swaps colors $(0,1)$ on $L_0$ and does nothing on $P_0\setminus L_0$.
Thus after step (1) one has
\[
(f_0,f_1,f_2)=
\begin{cases}
(\mathrm{bump}_j,\mathrm{bump}_i,\mathrm{bump}_k), & \text{on }L_0,\\
(\mathrm{bump}_i,\mathrm{bump}_j,\mathrm{bump}_k), & \text{on }P_0\setminus L_0.
\end{cases}
\]
Step (2) swaps colors $(0,2)$ on all of $P_0$, so this becomes
\[
(f_0,f_1,f_2)=
\begin{cases}
(\mathrm{bump}_k,\mathrm{bump}_i,\mathrm{bump}_j), & \text{on }L_0,\\
(\mathrm{bump}_k,\mathrm{bump}_j,\mathrm{bump}_i), & \text{on }P_0\setminus L_0.
\end{cases}
\]
Step (3) swaps colors $(0,1)$ on all of $P_0$, yielding
\[
(f_0,f_1,f_2)=
\begin{cases}
(\mathrm{bump}_i,\mathrm{bump}_k,\mathrm{bump}_j), & \text{on }L_0,\\
(\mathrm{bump}_j,\mathrm{bump}_k,\mathrm{bump}_i), & \text{on }P_0\setminus L_0.
\end{cases}
\]
Since $L_0=\{S=0,k=0\}$, this is exactly the stated rule on $S=0$.

On $P_1$, the first three steps do nothing.
Step (4) swaps colors $(0,1)$ on $L_1$, so after step (4)
\[
(f_0,f_1,f_2)=
\begin{cases}
(\mathrm{bump}_j,\mathrm{bump}_i,\mathrm{bump}_k), & \text{on }L_1,\\
(\mathrm{bump}_i,\mathrm{bump}_j,\mathrm{bump}_k), & \text{on }P_1\setminus L_1.
\end{cases}
\]
Step (5) swaps colors $(0,2)$ on all of $P_1$, so finally
\[
(f_0,f_1,f_2)=
\begin{cases}
(\mathrm{bump}_k,\mathrm{bump}_i,\mathrm{bump}_j), & \text{on }L_1,\\
(\mathrm{bump}_k,\mathrm{bump}_j,\mathrm{bump}_i), & \text{on }P_1\setminus L_1.
\end{cases}
\]
Since $L_1=\{S=1,k=0\}$, this is exactly the stated rule on $S=1$.
\end{proof}

\begin{proposition}[Explicit affine conjugacies to the odometer]\label{prop:odd-odometer}
Parameterize $P_0$ by $(i,k)\in \Zm^2$ with $j\equiv -i-k$.
For the coloring of Proposition~\ref{prop:odd-closed}, the return maps
\[
F_c=f_c^m\big|_{P_0}:P_0\to P_0
\]
satisfy
\[
F_0(i,k)=(i-2+\mathbf 1_{k=0},\,k+1),\qquad
F_1(i,k)=(i+\mathbf 1_{k=-1},\,k+1),
\]
\[
F_2(i,k)=(i+2-2\mathbf 1_{k=0},\,k-2).
\]
If $m$ is odd, then with
\[
\psi_0(i,k)=(k,\ i+2k),\qquad
\psi_1(i,k)=(k+1,\ i),
\]
and
\[
\psi_2(i,k)=\bigl(\lambda k,\ \lambda(i+k)\bigr),
\qquad \lambda:=(-2)^{-1}\in\Zm,
\]
one has
\[
\psi_c\circ F_c = O\circ \psi_c
\qquad (c=0,1,2).
\]
\end{proposition}

\begin{proof}
Each return from $P_0$ to itself consists of one step on layer $S=0$, one step on layer $S=1$, and then $m-2$ canonical steps on layers $S\ge 2$.
So the formulas for $F_0,F_1,F_2$ come from a one-return count using Proposition~\ref{prop:odd-closed}.

For color $0$, if $k=0$ then the layer-$0$ step uses direction $0$, the layer-$1$ step uses direction $2$, and the remaining $m-2$ steps use the canonical direction $0$.
This gives displacement $(m-1,1)\equiv(-1,1)$ in the $(i,k)$ coordinates.
If $k\neq 0$, the only change is that the layer-$0$ step uses direction $1$, so the displacement becomes $(m-2,1)\equiv(-2,1)$.
Hence
\[
F_0(i,k)=(i-2+\mathbf 1_{k=0},\,k+1).
\]
The formulas for $F_1$ and $F_2$ are obtained by the same return count:
for color $1$, the exceptional layer-$1$ branch occurs precisely when the first step changes $k$ from $-1$ to $0$;
for color $2$, the exceptional branch occurs precisely when the start point has $k=0$.

Now assume $m$ is odd, so $\lambda=(-2)^{-1}$ exists in $\Zm$.
The linear parts of $\psi_0,\psi_1,\psi_2$ have determinants $-1,-1,-\lambda^2$, hence all three maps are affine bijections of $\Zm^2$.
A direct calculation gives
\[
\psi_0(F_0(i,k))
=
(k+1,\ i+2k+\mathbf 1_{k=0})
=
O(\psi_0(i,k)),
\]
\[
\psi_1(F_1(i,k))
=
(k+2,\ i+\mathbf 1_{k=-1})
=
O(\psi_1(i,k)),
\]
and
\[
\psi_2(F_2(i,k))
=
\bigl(\lambda(k-2),\ \lambda(i+k-2\mathbf 1_{k=0})\bigr)
=
\bigl(\lambda k+1,\ \lambda(i+k)+\mathbf 1_{k=0}\bigr)
=
O(\psi_2(i,k)).
\]
Thus each $F_c$ is affine-conjugate to the standard odometer.
\end{proof}

\begin{theorem}[Odd case]\label{thm:odd-hamilton}
If $m$ is odd, the coloring of Proposition~\ref{prop:odd-closed} decomposes the arc set of $D_3(m)$ into three directed Hamilton cycles.
\end{theorem}

\begin{proof}
By Proposition~\ref{prop:odd-odometer} and Lemma~\ref{lem:odometer}, each return map $F_c$ is a single $m^2$-cycle on $P_0$.
Lemma~\ref{lem:return} then lifts each of these return maps to a directed Hamilton cycle on $V$.
Since Theorem~\ref{thm:5swap} already gives a valid coloring, the three color classes form a Hamilton decomposition.
\end{proof}

\begin{proposition}[Failure of color~$2$ for even $m$]\label{prop:even-shatter}
Let $m\ge 4$ be even and let $(f_0,f_1,f_2)$ be the affine coloring of Proposition~\ref{prop:odd-closed}.
Then $f_0$ and $f_1$ are Hamilton cycles, while $f_2$ decomposes into exactly $m+2$ directed cycles.
\end{proposition}

\begin{remark}
In return-map language the defect is now transparent.
Colors $0$ and $1$ retain a primitive clock, but color~$2$ loses it because the update $k\mapsto k-2$ splits $\mathbb Z_m$ into two parity classes before any carry can reconnect them.
The even construction will therefore modify only the low-layer directions that feed this clock, with the goal of restoring a primitive carry mechanism without breaking the other two colors.
\end{remark}

\begin{proof}[Proof of Proposition~\ref{prop:even-shatter}]
The formulas for $f_0$ and $f_1$ are unchanged from the odd case, so those two colors are still Hamilton.
For $f_2$, the return map is
\[
F_2(i,k)=(i+2-2\mathbf 1_{k=0},\,k-2).
\]
Because $m$ is even, the update $k\mapsto k-2$ preserves the parity of $k$, so $P_0$ splits into the odd-$k$ half and the even-$k$ half.

On the odd-$k$ half, the second coordinate runs through all odd residues with period $m/2$ and never hits $0$.
So over one full $k$-cycle the first coordinate gains $2\cdot(m/2)=m\equiv 0$, so every orbit there has length $m/2$.
Since the odd-$k$ half has size $m^2/2$, it contributes
\[
\frac{m^2/2}{m/2}=m
\]
cycles.

On the even-$k$ half, the second coordinate again has period $m/2$, but now it hits $0$ exactly once per such period.
Therefore over one full $k$-cycle the net change in $i$ is
\[
2\left(\frac m2-1\right)=m-2\equiv -2 \pmod m.
\]
So the induced return map after $m/2$ steps is translation by $-2$ on $\mathbb Z_m$, which has exactly $\gcd(m,2)=2$ cycles.
Consequently the even-$k$ half contributes $2$ cycles.

Altogether $f_2$ decomposes into $m+2$ directed cycles.
\end{proof}

\section{Even \texorpdfstring{$m$}{m}: repairing the clock with Route~E}\label{sec:even}

The odd construction shows what the finished picture should look like: after the layer drift is removed, the dynamics should reduce to a primitive clock together with its carry.
When $m$ is even, Corollary~\ref{cor:parity-barrier} rules out any Kempe-from-canonical derivation of such a coloring.
The even case is therefore a direct design problem.
We build a low-layer direction assignment, Route~E, whose return maps are finite-defect perturbations of the same odometer mechanism and whose remaining defects can be organized by one more first-return reduction to lane transversals.

The first step is to isolate the clock-and-carry criterion abstractly.
On $P_0\cong (\mathbb Z_m)^2$, a return map of the form
\[
(i,k)\longmapsto (i+\alpha(k),\ k+d)
\]
contains two pieces of data: a clock step $d$ and a carry profile $\alpha$.
If the clock is primitive and the total carry is primitive, then the whole map is a single cycle.
That is the structural target that Route~E is designed to recover.

\begin{lemma}[Clock-and-carry criterion]\label{lem:skew-product}
Fix $m\ge 1$, a step $d\in \mathbb Z_m$ with $\gcd(d,m)=1$, and a function $\alpha:\mathbb Z_m\to\mathbb Z_m$.
Define
\[
F(i,k):=(i+\alpha(k),\,k+d).
\]
Then $F$ is a permutation and
\[
F^m(i,k)=(i+\Delta,\,k),\qquad \Delta:=\sum_{t\in\mathbb Z_m}\alpha(t).
\]
Consequently, every orbit of $F$ has length $m^2/\gcd(\Delta,m)$, so $F$ is a single $m^2$-cycle iff $\gcd(\Delta,m)=1$.
\end{lemma}

\begin{proof}
Since $d$ is a unit, $k\mapsto k+d$ is an $m$-cycle and $F$ is bijective with inverse $(i,k)\mapsto(i-\alpha(k-d),k-d)$.
Iterating gives
\[
F^t(i,k)=\Bigl(i+\sum_{r=0}^{t-1}\alpha(k+rd),\,k+td\Bigr).
\]
At $t=m$ the second coordinate returns to $k$ and the sum ranges over all residues of $\mathbb Z_m$, giving $\Delta$.
The orbit length follows from the order of translation by $\Delta$ on the first coordinate.
\end{proof}

For the affine odd construction, Lemma~\ref{lem:skew-product} simply repackages the odometer story in linear form: the second coordinate is the clock and the first accumulates the carry.
When $m$ is even, color~$2$ fails precisely because the clock step $-2$ is not a unit.
So the problem is not to search blindly for another coloring, but to repair the low layers until a primitive clock and a primitive total carry reappear, while keeping the defect set finite and geometrically explicit.

The sign-product barrier shows that no Kempe-from-canonical derivation can do this in even modulus.
Accordingly, we introduce the Route~E \emph{direction assignment} on the low layers $S\in\{0,1,2\}$ and keep all higher layers canonical.

For $v=(i,j,k)\in V$, write $S(v)=i+j+k\pmod m$.
A direction triple $(d_0,d_1,d_2)\in\{0,1,2\}^3$ means that color $c$ uses the arc $v\mapsto v+e_{d_c}$.
At this stage we specify only a direction assignment in the sense of Definition~\ref{def:direction-assignment}; the fact that the associated maps are permutations, and hence form a coloring, is proved later in Theorem~\ref{thm:routeE-even-main}.

\begin{definition}[Route~E layer-$0$ defect families]\label{def:routeE-layer0-families}
Let $m\ge 6$ be even.
On the layer $\{S=0\}$ define the following affine defect families; the subscripts record the direction triple assigned there.

\smallskip
\noindent\textbf{Case I: $m\equiv 0,2\pmod 6$.}
\begin{align*}
X_{102}&:=\{(0,0,0)\}\cup\{(i,1,m-1-i):1\le i\le m-3\}\cup\{(m-1,2,m-1)\},\\
X_{021}&:=\{(0,1,m-1)\}\cup\{(i,m-i,0):1\le i\le m-3\}\cup\{(m-1,0,1)\},\\
X_{210}&:=\{(0,j,m-j):2\le j\le m-1\}\cup\{(1,0,m-1)\}.
\end{align*}

\smallskip
\noindent\textbf{Case II: $m\equiv 4\pmod 6$.}
\begin{align*}
Y_{102}&:=\{(0,0,0)\}\cup\{(i,1,m-1-i):2\le i\le m-3\}\cup\{(m-1,2,m-1)\},\\
Y_{021}&:=\{(0,1,m-1)\}\cup\{(i,m-i,0):2\le i\le m-3\}\cup\{(m-1,0,1)\},\\
Y_{210}&:=\{(0,j,m-j):2\le j\le m-1\}\cup\{(1,0,m-1)\}\\
&\hspace{2em}\cup\{(1,j,m-1-j):2\le j\le m-2\}\cup\{(2,0,m-2),(2,m-1,m-1)\}.
\end{align*}
\end{definition}

\begin{definition}[Explicit Route~E direction assignment]\label{def:routeE}
Let $m\ge 6$ be even. Define a direction triple $\delta(v)=(d_0(v),d_1(v),d_2(v))$ as follows.
The layer-by-layer structure is summarized below; the formal rules appear in the subsequent paragraphs.

\smallskip
\begin{center}
\footnotesize
\begin{tabular}{@{}lll@{}}
\toprule
layer & rule & complexity \\
\midrule
$S\ge 3$ & canonical $(0,1,2)$ & none \\
$S=2$ & two branches, split on $j=0$ vs.\ $j\ne 0$ & $2$ triples \\
$S=1$ & two branches, split on $i=0$ vs.\ $i\ne 0$ & $2$ triples \\
$S=0$ & six branches, depending on $m\bmod 6$ & $\le 6$ triples \\
\bottomrule
\end{tabular}
\end{center}
\smallskip

\noindent
The design principle is as follows.
Layers $S\ge 3$ keep the canonical triple, so the return map on $P_0$ sees their contribution as a uniform translation.
Layers $1$ and $2$ each introduce one simple branch: they test whether a single coordinate vanishes after the preceding step and swap two directions if so.
Layer $0$ contains the nondefault defect scaffold.
In Case~I this scaffold is exactly the family data from Definition~\ref{def:routeE-layer0-families}; when $m\equiv 4\pmod 6$, the $Y_{210}$ enlargement supplies the Case~II repair family.
That repair family changes the interior height ordering for colors $1$ and $0$, breaks the residue-$3$ obstruction that persists in the primary geometry, and leaves only a finite splice problem on the resulting arithmetic family-blocks.

\paragraph{Layers $S\notin\{0,1,2\}$.}
Use the canonical triple $(d_0,d_1,d_2)=(0,1,2)$.

\paragraph{Layer $S=1$.}
Use
\[
(d_0,d_1,d_2)=
\begin{cases}
(1,0,2), & i=0,\\
(2,0,1), & i\ne 0.
\end{cases}
\]

\paragraph{Layer $S=2$.}
Use
\[
(d_0,d_1,d_2)=
\begin{cases}
(2,1,0), & j=0,\\
(0,1,2), & j\ne 0.
\end{cases}
\]

\paragraph{Layer $S=0$, case $m\equiv 0,2\pmod 6$.}
Use
\[
(d_0,d_1,d_2)=
\begin{cases}
(1,0,2), & v\in X_{102},\\
(0,2,1), & v\in X_{021},\\
(2,1,0), & v\in X_{210},\\
(0,1,2), & v=(m-2,1,1),\\
(2,0,1), & v=(m-2,2,0),\\
(1,2,0), & \text{otherwise.}
\end{cases}
\]
\smallskip\noindent
where $X_{102},X_{021},X_{210}$ are the Case~I layer-$0$ defect families from Definition~\ref{def:routeE-layer0-families}.

\paragraph{Layer $S=0$, case $m\equiv 4\pmod 6$.}
Use
\[
(d_0,d_1,d_2)=
\begin{cases}
(1,0,2), & v\in Y_{102},\\
(0,2,1), & v\in Y_{021},\\
(2,1,0), & v\in Y_{210},\\
(0,1,2), & v\in\{(1,1,m-2),(m-2,1,1)\},\\
(2,0,1), & v\in\{(1,m-1,0),(m-2,2,0)\},\\
(1,2,0), & \text{otherwise.}
\end{cases}
\]
where $Y_{102},Y_{021},Y_{210}$ are the Case~II layer-$0$ defect families from Definition~\ref{def:routeE-layer0-families}.
This completes the definition: set $f_c(v)=v+e_{d_c(v)}$; these are the associated color maps whose $m$-step behaviour on $P_0$ is analyzed below.
\end{definition}

\begin{remark}[Geometric summary of the layer-$0$ families]\label{rem:routeE-layer0-geometry}
The default layer-$0$ triple is $(1,2,0)$; the named families are the thin affine tracks where this default is overridden.
In Case~I, $X_{102},X_{021},X_{210}$ are, up to endpoint corrections, the three lines $j=1$, $k=0$, and $i=0$ inside $P_0$.
Thus the primary geometry consists of three one-dimensional defect tracks in an otherwise uniform background.
In Case~II, the first two tracks are shortened near $i=1$, while the third track is enlarged by the adjacent affine line $i=1$ together with two boundary points.
We call that enlargement of the third track the \emph{Case~II repair family}.
It is the only new layer-$0$ ingredient in Case~II, and later sections show that it changes the interior return dynamics rather than merely patching boundary points.
After parameterizing $P_0$ by $(i,k)$ in Appendix~\ref{app:routeE-derivation}, these same tracks become the sets $A,B,C$ together with the isolated points $D,E$.
\end{remark}

\begin{lemma}[Layer-$0$ partition]\label{rem:layer0-partition}
In each parity class ($m\equiv 0,2\pmod 6$ and $m\equiv 4\pmod 6$), the listed special sets ($X_{102},X_{021},X_{210}$ or $Y_{102},Y_{021},Y_{210}$), the isolated exceptional vertices, and the default branch are pairwise disjoint and together partition the layer $\{S=0\}$.
Hence Definition~\ref{def:routeE} assigns exactly one direction triple to every vertex of $D_3(m)$.
\end{lemma}

\begin{proof}
Exhaustiveness is immediate from the final ``otherwise'' clause in each layer-$0$ rule, so only disjointness needs checking.
In Case~I, away from their endpoint corrections, the three families lie on the affine lines $j=1$, $k=0$, and $i=0$ inside $P_0$.
The possible pairwise intersections of those three lines in $P_0$ are $(0,1,m-1)$, $(1,0,m-1)$, and $(m-1,1,0)$.
By construction, the first belongs only to $X_{021}$, the second only to $X_{210}$, and the third to none of $X_{102},X_{021},X_{210}$.
The remaining listed exceptional vertices $(0,0,0)$, $(m-1,2,m-1)$, $(m-1,0,1)$, $(m-2,1,1)$, and $(m-2,2,0)$ are outside the other special families by inspection of the coordinate ranges.
In Case~II, the same endpoint conventions separate $Y_{102}\subseteq\{j=1\}$, $Y_{021}\subseteq\{k=0\}$, and the original part of $Y_{210}\subseteq\{i=0\}$, while the added Case~II repair family lies on $i=1$ together with the two boundary points $(2,0,m-2)$ and $(2,m-1,m-1)$.
The ranges $2\le i\le m-3$ and $2\le j\le m-2$ exclude any overlap of that added $i=1$ track with $Y_{102}$ or $Y_{021}$, and the two added boundary points are also outside $Y_{102}\cup Y_{021}$.
Thus the listed special sets are pairwise disjoint in both parity classes.
\end{proof}

Once stated, the construction is verified purely combinatorially.

\begin{definition}[Stall]\label{def:stall}
A \emph{stall} is a step at which the return map departs from the generic (uniform-translation) branch.
\end{definition}

\noindent
The number of stalls per orbit is bounded independently of~$m$ (see the orbit traces in Appendix~\ref{app:routeE-cycle}).

Below we develop the return-map framework (\S\ref{sec:framework}), prove color~$2$ in full (\S\ref{sec:even-r2}), and assemble the main even-case theorem.
Example~\ref{ex:m6-color2} illustrates the full machinery on the smallest admissible case $m=6$.

\subsection{Framework for the Route~E return maps}\label{sec:framework}

The appendices contain the detailed tables and the longer case-by-case bookkeeping, but the logical mechanism of the even proof can be stated in the main text.
Conceptually, the argument proceeds in four steps:
\begin{enumerate}[label=\textup{(\roman*)},leftmargin=2.5em]
\item a primary geometry on the low layers;
\item a precise obstruction in that geometry;
\item a short no-go argument showing that the remaining Case~II boundary corrections alone do not work;
\item the actual repaired Route~E construction followed by a finite splice.
\end{enumerate}
The first input is a three-step transducer extracted from Definition~\ref{def:routeE}.

\begin{proposition}[Three-step transducer reduction]\label{prop:routeE-transducer}
Parameterize
\[
P_0=\{(i,j,k)\in \mathbb Z_m^3:i+j+k\equiv 0\}
\]
by
\[
v(i,k):=(i,-i-k,k),\qquad (i,k)\in \mathbb Z_m^2.
\]
For a color $c\in\{0,1,2\}$ and a start point $v(i,k)\in P_0$, write the first three color-$c$ directions on the layers $S=0,1,2$ as
\[
\omega_c(i,k)=(a_c,b_c,c_c)\in\{0,1,2\}^3.
\]
Then the return maps satisfy
\[
R_0(i,k)=\bigl(i+N_0(\omega_0)-3,\ k+N_2(\omega_0)\bigr),
\]
\[
R_1(i,k)=\bigl(i+N_0(\omega_1),\ k+N_2(\omega_1)\bigr),
\]
\[
R_2(i,k)=\bigl(i+N_0(\omega_2),\ k+N_2(\omega_2)-3\bigr),
\]
where $N_r(\omega)$ counts the occurrences of direction $r$ in the word $\omega$.
Moreover:
\begin{enumerate}[label=\textup{(\roman*)},leftmargin=2.5em]
\item for color $1$, one always has $b_1=0$ and $c_1=1$;
\item for colors $0$ and $2$, the layer-$1$ direction depends only on whether the post-step-$1$ $i$-coordinate equals $0$, and the layer-$2$ direction depends only on whether the post-step-$2$ $j$-coordinate equals $0$.
\end{enumerate}
Appendix~\ref{app:routeE-derivation} therefore reduces the derivation of the piecewise formulas for $R_0,R_1,R_2$ to a finite table evaluation on the layer-$0$ partition.
\end{proposition}

\begin{proof}
After the first three steps from $P_0$, the orbit lies on layer $S=3$.
All subsequent layers are canonical, so until the next return to $P_0$ the map follows direction $c$ for exactly $m-3$ further steps.
This yields the displacement formulas above.

Now write $j=-i-k$ on $P_0$.
After step $1$, the $i$-coordinate is
\[
i_1=i+\mathbf 1_{a_c=0}.
\]
By Definition~\ref{def:routeE}, on layer $S=1$ the color-$1$ direction is always $0$, the color-$0$ direction is $1$ if $i_1=0$ and $2$ otherwise, and the color-$2$ direction is $2$ if $i_1=0$ and $1$ otherwise.
So part~\textup{(i)} is immediate, and for colors $0$ and $2$ the second letter $b_c$ is determined by the single test $i_1=0$.

After step $2$, the $j$-coordinate is
\[
j_2=j+\mathbf 1_{a_c=1}+\mathbf 1_{b_c=1}.
\]
Again by Definition~\ref{def:routeE}, on layer $S=2$ the color-$0$ direction is $2$ if $j_2=0$ and $0$ otherwise, while the color-$2$ direction is $0$ if $j_2=0$ and $2$ otherwise.
Thus for colors $0$ and $2$ the third letter $c_c$ is determined by the single test $j_2=0$.
The appendix tables simply evaluate these tests on the finite partition of $P_0$ induced by the layer-$0$ rule.
\end{proof}

\begin{lemma}[Generic three-step words and bulk vectors]\label{lem:routeE-generic}
Consider the Route~E direction assignment.
Suppose a start point $(i,k)\in P_0\cong\Zm^2$ has the base layer-$0$ triple $(1,2,0)$ and that, along the first two steps of the orbit, neither test $i_1=0$ on layer $1$ nor $j_2=0$ on layer $2$ is triggered.
Then the three low-layer words are constant:
\[
\omega_0(i,k)=120,\qquad \omega_1(i,k)=201,\qquad \omega_2(i,k)=012.
\]
Consequently the return maps follow the bulk branches
\[
R_0(i,k)=(i-2,k+1),\qquad R_1(i,k)=(i+1,k+1),\qquad R_2(i,k)=(i+1,k-2).
\]
\end{lemma}

\begin{proof}
The base layer-$0$ triple $(1,2,0)$ means that on layer $0$ the colors $0,1,2$ use directions $1,2,0$, respectively.
If the post-step-$1$ test $i_1=0$ is not triggered, then the layer-$1$ triple is the generic one $(2,0,1)$.
If the post-step-$2$ test $j_2=0$ is not triggered, then the layer-$2$ triple is the generic one $(0,1,2)$.
Reading off the color-$c$ letters from these three generic triples yields the words $120,201,012$.
The displacement formulas now follow immediately from Proposition~\ref{prop:routeE-transducer}.
\end{proof}

\begin{remark}[structure of the defect sets]\label{rem:routeE-design}
Lemma~\ref{lem:routeE-generic} explains the bulk of the construction: almost everywhere the low-layer word is constant, so the return maps are translations by the three displayed bulk vectors.
Non-generic behaviour can occur only for three reasons:
\begin{enumerate}[label=\textup{(\alph*)},leftmargin=2em]
\item the start point already lies on one of the special layer-$0$ defect lines or endpoint corrections;
\item after the first step, the test $i_1=0$ is triggered on layer $1$;
\item after the second step, the test $j_2=0$ is triggered on layer $2$.
\end{enumerate}
In the $(i,k)$ coordinates on $P_0$, each of these conditions is affine, which is why the defect support in Lemma~\ref{lem:routeE-finite-defect} consists of affine lines plus finitely many isolated boundary points.
The boundary-point exceptions are tuned to cut the long arithmetic runs---maximal runs of consecutive lanes that share the same bulk itinerary---at finitely many bridge vertices and to record the finite splice permutation on the resulting ordered arithmetic family-blocks.
In particular, once the first-return formulas are known, the single-cycle question for the induced lane map reduces to the cycle structure of that finite splice permutation.
What changes at $m\equiv 4\pmod{6}$, however, is not only the last splice step.
In the natural primary geometry obtained by extending the Case~I layer-$0$ families to all even $m$, colors $0$ and $2$ already close, while color $1$ remains trapped in three residue-$3$ strands.
The Case~II repair family changes the interior height ordering for colors $1$ and $0$ before the final splice; Propositions~\ref{prop:primary-geometry-obstruction} and~\ref{prop:boundary-no-go} below make this necessity precise.
\end{remark}

The next lemma is the counting lemma used in all three color proofs.

\begin{lemma}[First-return counting]\label{lem:counting}
Let $F:X\to X$ be a map on a finite set $|X|=N$, and let $L\subseteq X$.
Assume that for every $x\in L$ the first return time
\[
\rho(x)=\min\{t\ge 1:F^t(x)\in L\}
\]
exists, and write
\[
T(x)=F^{\rho(x)}(x)\in L
\]
for the induced first-return map.
If
\begin{enumerate}[label=\textup{(\roman*)}]
\item $T$ is a single $|L|$-cycle on $L$, and
\item $\sum_{x\in L}\rho(x)=N$,
\end{enumerate}
then $F$ is a single $N$-cycle on $X$.
\end{lemma}

\begin{remark}
Lemma~\ref{lem:counting} is the return-map principle applied a second time.
On $P_0$ the state space has already dropped from size $m^3$ to size $m^2$, but the bulk dynamics are still almost pure clock motion.
The natural next step is therefore to cut once more to explicit lane transversals and study only the induced carry.
That is the point at which the even proof becomes one-dimensional.
\end{remark}

\begin{proof}
Pick $x_0\in L$.
Because $T$ is a cycle on all of $L$, the return points along the $F$-orbit of $x_0$ are
\[
x_0,\ T(x_0),\ T^2(x_0),\ \dots,\ T^{|L|-1}(x_0),\ x_0,
\]
and every point of $L$ appears exactly once before the sequence closes.

Set
\[
N_0:=\sum_{r=0}^{|L|-1}\rho\bigl(T^r(x_0)\bigr).
\]
By construction, after $N_0$ applications of $F$ one returns to $x_0$.
So the $F$-orbit of $x_0$ is a cycle of length $N_0$.
The hypothesis gives
\[
N_0=\sum_{x\in L}\rho(x)=N.
\]
Hence one $F$-cycle already has size $N=|X|$, so it contains all points of $X$.
Thus $F$ is a single $N$-cycle.
\end{proof}

For each color we now choose coordinates adapted to the \emph{bulk branch} of the return map, meaning the branch that applies away from a bounded defect set.
In those coordinates the generic step (also called the \emph{bulk} step and denoted $G$ in the orbit displays) advances one coordinate---the \emph{clock}---uniformly while preserving the other---the \emph{lane}.
This is again the odometer picture: the clock turns every time, while the lane changes only at the rare non-generic events.
Those non-generic events are the \emph{stalls} introduced above.
At a stall the lane shifts by a bounded amount, and the clock may gain or lose a bounded correction.
In the orbit displays below, $\xrightarrow{\bulk}$ or $\xrightarrow{\bulk^t}$ denotes one or $t$ consecutive generic steps, while $\xrightarrow{+1}$ or $\xrightarrow{+2}$ denotes a stall that raises the lane by $1$ or $2$.

Because the number of stalls per orbit is bounded independently of $m$, the two-dimensional dynamics on $P_0$ reduce to one-dimensional first-return maps on explicit transversals.
On those transversals, most lanes follow one arithmetic rule and only $O(1)$ boundary lanes require special treatment.

\begin{lemma}[Bulk coordinates expose the clock]\label{lem:routeE-bulkcoords}
Define linear maps $\Phi_c:P_0\cong \Zm^2\to \Zm^2$ by
\[
\Phi_1(i,k)=(i-k,\ k),
\]
\[
\Phi_2(i,k)=(2i+k,\ -i-k),
\]
\[
\Phi_0(i,k)=(i+2k,\ k).
\]
Let
\[
b_1=(1,1),\qquad b_2=(1,-2),\qquad b_0=(-2,1).
\]
Then each $\Phi_c$ is invertible over $\Zm$, and in the corresponding $(u,t)$-coordinates the bulk move $z\mapsto z+b_c$ becomes
\[
(u,t)\mapsto (u,t+1).
\]
\end{lemma}

\begin{proof}
The matrices of $\Phi_1,\Phi_2,\Phi_0$ are
\[
M_1=\begin{pmatrix}1&-1\\0&1\end{pmatrix},\qquad
M_2=\begin{pmatrix}2&1\\-1&-1\end{pmatrix},\qquad
M_0=\begin{pmatrix}1&2\\0&1\end{pmatrix}.
\]
Their determinants are $1,-1,1$, so each is a unit in $\Zm$ and each map is invertible.
A direct multiplication gives
\[
M_1b_1=(0,1),\qquad M_2b_2=(0,1),\qquad M_0b_0=(0,1),
\]
which is exactly the claim.
\end{proof}

\begin{lemma}[Finite-defect odometer normal form for Route~E]\label{lem:routeE-finite-defect}
For each color $c\in\{0,1,2\}$, let
\[
\widehat R_c:=\Phi_c\circ R_c\circ \Phi_c^{-1}.
\]
Then there is a defect set $E_c\subseteq \Zm^2$, contained in a finite union of affine lines together with finitely many isolated points, such that
\[
\widehat R_c(u,t)=(u,t+1)\qquad\text{for }(u,t)\notin E_c.
\]
More precisely:
\begin{enumerate}[label=\textup{(\roman*)},leftmargin=2.5em]
\item For color $1$, every defect step is one of
\[
(u,t)\mapsto (u+1,t),\qquad (u,t)\mapsto (u+2,t),
\]
and the defect set is contained in
\[
u+t=0,\qquad u+2t=m-1,
\]
and, when $m\equiv 4\pmod 6$, also
\[
u+t=1,
\]
together with finitely many isolated points.

\item For color $2$, every defect step is one of
\[
(u,t)\mapsto (u-1,t+2),\quad (u,t)\mapsto (u+1,t+1),\quad (u,t)\mapsto (u-1,t+1),
\]
\[
(u,t)\mapsto (u-2,t+2),\quad (u,t)\mapsto (u+1,t),
\]
and the defect set is contained in
\[
t=1,\qquad t=m-1,\qquad u+2t=0,\qquad u+t=m-1,
\]
together with finitely many isolated points.

\item For color $0$, every defect step is one of
\[
(u,t)\mapsto (u-2,t),\quad (u,t)\mapsto (u+3,t+3),\quad (u,t)\mapsto (u+2,t+2),
\]
\[
(u,t)\mapsto (u-1,t),\quad (u,t)\mapsto (u+1,t+2),\quad (u,t)\mapsto (u+1,t+1),
\]
and the defect set is contained in
\[
t=0,\qquad u=t+1,
\]
and, when $m\equiv 4\pmod 6$, also
\[
u=1+2t,
\]
together with finitely many isolated points.
\end{enumerate}
\end{lemma}

\begin{proof}
We compute directly from the piecewise formulas for $R_0,R_1,R_2$ recorded in Appendix~\ref{app:routeE}.

For color $1$, the generic branch is $(i,k)\mapsto(i+1,k+1)$, while the two exceptional branches are $(i,k)\mapsto(i+1,k)$ and $(i,k)\mapsto(i+2,k)$.
Under $\Phi_1(i,k)=(u,t)=(i-k,k)$ these become
\[
(i+1,k+1)\mapsto (u,t+1),\qquad (i+1,k)\mapsto (u+1,t),\qquad (i+2,k)\mapsto (u+2,t).
\]
The supporting conditions in Appendix~\ref{app:routeE} are $i=0$, $i=1$, and $i+k\equiv m-1$, plus finitely many exceptional boundary points.
These are exactly the affine conditions
\[
i=0 \iff u+t=0,\qquad i=1 \iff u+t=1,\qquad i+k\equiv m-1 \iff u+2t=m-1.
\]
This proves \textup{(i)}.

For color $2$, the generic branch is $(i,k)\mapsto(i+1,k-2)$.
The exceptional branches listed in Appendix~\ref{app:routeE} are
\[
(i+1,k-3),\qquad (i+2,k-3),\qquad (i,k-1),\qquad (i,k-2),\qquad (i+1,k-1).
\]
Under $\Phi_2(i,k)=(u,t)=(2i+k,-i-k)$ one obtains
\[
(i+1,k-2)\mapsto (u,t+1),
\]
\[
(i+1,k-3)\mapsto (u-1,t+2),\qquad (i+2,k-3)\mapsto (u+1,t+1),
\]
\[
(i,k-1)\mapsto (u-1,t+1),\qquad (i,k-2)\mapsto (u-2,t+2),\qquad (i+1,k-1)\mapsto (u+1,t).
\]
The supporting conditions in Appendix~\ref{app:routeE} are $i+k\equiv m-1$, $i+k\equiv 1$, $k=0$, and $i=m-1$, together with finitely many exceptional points.
These become
\[
i+k\equiv m-1 \iff t=1,
\]
\[
i+k\equiv 1 \iff t=m-1,
\]
\[
k=0 \iff u+2t=0,
\]
\[
i=m-1 \iff u+t=m-1.
\]
This proves \textup{(ii)}.

For color $0$, the generic branch is $(i,k)\mapsto(i-2,k+1)$.
The exceptional branches listed in Appendix~\ref{app:routeE} are
\[
(i-2,k),\qquad (i-3,k+3),\qquad (i-2,k+2),\qquad (i-1,k),\qquad (i-3,k+2),\qquad (i-1,k+1).
\]
Under $\Phi_0(i,k)=(u,t)=(i+2k,k)$ these become
\[
(i-2,k+1)\mapsto (u,t+1),
\]
\[
(i-2,k)\mapsto (u-2,t),\qquad (i-3,k+3)\mapsto (u+3,t+3),\qquad (i-2,k+2)\mapsto (u+2,t+2),
\]
\[
(i-1,k)\mapsto (u-1,t),\qquad (i-3,k+2)\mapsto (u+1,t+2),\qquad (i-1,k+1)\mapsto (u+1,t+1).
\]
The supporting conditions in Appendix~\ref{app:routeE} are $k=0$, $i+k\equiv 1$, and, in Case~II, $i=1$, together with finitely many exceptional points.
These become
\[
k=0 \iff t=0,
\]
\[
i+k\equiv 1 \iff u=t+1,
\]
\[
i=1 \iff u=1+2t.
\]
This proves \textup{(iii)}.
\end{proof}

\begin{remark}[origin of the Case~I\,/\,II split]
\label{rem:mod6-split}
The distinction between Case~I ($m\equiv 0,2\pmod 6$) and Case~II ($m\equiv 4\pmod 6$) that recurs in the color-by-color proofs below originates here.
When $m\equiv 0,2\pmod 6$ the primary defect support already governs the first-return dynamics.
When $m\equiv 4\pmod{6}$ the extra affine lines from Case~II are present---for color~$1$, the line $u+t=1$, and for color~$0$, the line $u=1+2t$---and they do more than add a few endpoint corrections.
They change the interior defect-height ordering on the transversal, which turns the primary color-$1$ drift $x\mapsto x+3$ into the repaired $+2/+6$ pattern and gives the uniform $+4$ interior drift for color~$0$; the later $m\bmod 12$ split for color~$0$ is then only a cyclic reordering of the same residue-$4$ splice picture.
\end{remark}

\begin{lemma}[defect-itinerary reduction]\label{lem:routeE-itinerary}
Let $F:\Zm^2\to\Zm^2$ be a map of the form
\[
F(z)=
\begin{cases}
G(z):=(u,t+1), & z\notin E,\\
D_\alpha(z):=(u+\Delta_\alpha,\ t+\tau_\alpha), & z\in E_\alpha,
\end{cases}
\]
where $E=\bigsqcup_\alpha E_\alpha$ is a finite disjoint union of defect components and each $D_\alpha$ has fixed increment $(\Delta_\alpha,\tau_\alpha)$.
Let
\[
L:=\{(u,0):u\in\Zm\}.
\]
Assume every point of $L$ has a first return to $L$.
Then for every start $(u,0)\in L$, the orbit segment before first return has a unique factorization
\[
G^{a_0}D_{\alpha_1}G^{a_1}D_{\alpha_2}\cdots D_{\alpha_r}G^{a_r},
\]
where each $a_j\ge 0$ and every point of each generic block lies outside $E$.
If
\[
F^{\rho(u)}(u,0)=(T(u),0),
\]
then
\[
T(u)=u+\sum_{j=1}^r \Delta_{\alpha_j}\pmod m,
\qquad
\rho(u)=\sum_{j=0}^r a_j+r.
\]
\end{lemma}

\begin{proof}
Because $F$ is deterministic, once an orbit point lies outside $E$ it is forced to follow the generic branch $G$ until it either returns to $L$ or first meets one of the defect components.
Starting from $(u,0)$, let $a_0$ be the maximal integer such that
\[
(u,0),\ G(u,0),\ \dots,\ G^{a_0-1}(u,0)
\]
all lie outside $E$.
If $G^{a_0}(u,0)\in L$, then there are no defect steps and the claim is immediate.
Otherwise $G^{a_0}(u,0)$ lies in a unique defect component $E_{\alpha_1}$, so the next step is forced to be $D_{\alpha_1}$.
Repeating the same argument recursively yields the factorization above, and uniqueness follows because there is never any branch choice.

The generic branch $G$ leaves the $u$-coordinate unchanged, while each defect step $D_{\alpha_j}$ changes $u$ by exactly $\Delta_{\alpha_j}$.
Therefore the net lane change from $(u,0)$ to $(T(u),0)$ is
\[
\sum_{j=1}^r \Delta_{\alpha_j},
\]
which proves the formula for $T(u)$.
The total number of steps is the number of generic steps plus the number of defect steps, giving
\[
\rho(u)=\sum_{j=0}^r a_j+r.
\]
\end{proof}

Combining Lemma~\ref{lem:routeE-finite-defect} with Lemma~\ref{lem:routeE-itinerary} reduces each Route~E cycle proof to a one-dimensional first-return problem on the transversal $L=\{(u,0):u\in\Zm\}$: one lists the defect components encountered by the bulk ray from $(u,0)$ before first return, reads off the lane increment and return time, and checks that the induced map $T$ is a single $m$-cycle with $\sum_u\rho(u)=m^2$.
Because the Route~E defect support is affine, the ordered defect itinerary is constant on long arithmetic families of start lanes; these are the arithmetic family-blocks later spliced in Proposition~\ref{prop:routeE-splice} (see \S\ref{app:routeE-cycle} for the formal block decompositions).
In Case~II the Case~II repair family changes this height ordering rather than merely inserting a few endpoint exceptions.

\begin{table}[t]
\centering
\footnotesize
\setlength{\tabcolsep}{4pt}
\begin{tabularx}{\textwidth}{@{}l l l >{\raggedright\arraybackslash}X@{}}
\toprule
color & bulk coordinates $(u,t)$ & bulk branch & defect support / transversal \\
\midrule
$1$ & $(i-k,\,k)$ & $(u,t)\mapsto(u,t+1)$ & defects on $u+t=0$, $u+2t=m-1$, and in Case~II also $u+t=1$; transversal $t=0$ \\
$2$ & $(2i+k,\,-i-k)$ & $(u,t)\mapsto(u,t+1)$ & defects on $t=1$, $t=m-1$, $u+2t=0$, $u+t=m-1$; transversal $t=0$ \\
$0$ & $(i+2k,\,k)$ & $(u,t)\mapsto(u,t+1)$ & defects on $t=0$, $u=t+1$, and in Case~II also $u=1+2t$; transversal $t=0$ \\
\bottomrule
\end{tabularx}
\caption{Route~E at a glance: adapted bulk coordinates, generic branch, and affine defect support for the three colors, derived from Lemmas~\ref{lem:routeE-bulkcoords} and~\ref{lem:routeE-finite-defect}. The main text works out color $2$ explicitly, while the appendix computes the remaining mod-$6$-dependent first-return data on the listed transversals.}
\label{tab:routeE-summary}
\end{table}

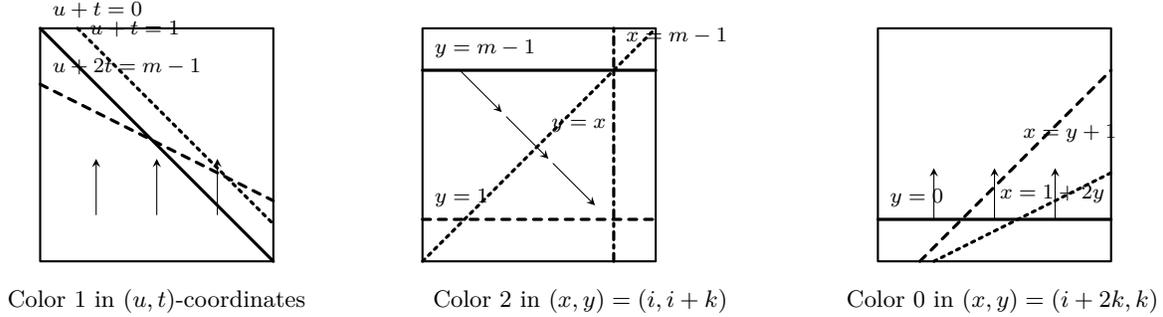
\begin{figure}[t]
\centering
\begin{minipage}[t]{0.32\textwidth}
\centering
\begin{tikzpicture}[x=0.62cm,y=0.62cm,>=stealth,line cap=round,line join=round]
\draw[thick] (0,0) rectangle (5,5);
\foreach \x in {1.2,2.5,3.8} {\draw[->] (\x,1.0) -- (\x,2.2);}
\draw[very thick] (0,5) -- (5,0);
\draw[very thick,dashed] (0,3.8) -- (5,1.3);
\draw[very thick,dotted] (0.8,5) -- (5,0.8);
\node[font=\scriptsize,anchor=south west] at (0.05,5) {$u+t=0$};
\node[font=\scriptsize,anchor=south west] at (0.05,3.8) {$u+2t=m-1$};
\node[font=\scriptsize,anchor=south west] at (0.85,4.6) {$u+t=1$};
\end{tikzpicture}

\smallskip
{\footnotesize Color $1$ in $(u,t)$-coordinates}
\end{minipage}\hfill
\begin{minipage}[t]{0.32\textwidth}
\centering
\begin{tikzpicture}[x=0.62cm,y=0.62cm,>=stealth,line cap=round,line join=round]
\draw[thick] (0,0) rectangle (5,5);
\foreach \x/\y in {0.8/4.1,1.8/3.1,2.8/2.1} {\draw[->] (\x,\y) -- ++(0.9,-0.9);}
\draw[very thick] (0,4.1) -- (5,4.1);
\draw[very thick,dashed] (0,0.9) -- (5,0.9);
\draw[very thick,dotted] (0,0) -- (5,5);
\draw[very thick,dash dot] (4.1,0) -- (4.1,5);
\node[font=\scriptsize,anchor=south west] at (0.05,4.15) {$y=m-1$};
\node[font=\scriptsize,anchor=south west] at (0.05,0.95) {$y=1$};
\node[font=\scriptsize,anchor=south west] at (2.55,2.55) {$y=x$};
\node[font=\scriptsize,anchor=south west] at (4.15,4.45) {$x=m-1$};
\end{tikzpicture}

\smallskip
{\footnotesize Color $2$ in $(x,y)=(i,i+k)$}
\end{minipage}\hfill
\begin{minipage}[t]{0.32\textwidth}
\centering
\begin{tikzpicture}[x=0.62cm,y=0.62cm,>=stealth,line cap=round,line join=round]
\draw[thick] (0,0) rectangle (5,5);
\foreach \x in {1.2,2.5,3.8} {\draw[->] (\x,0.9) -- (\x,2.0);}
\draw[very thick] (0,0.9) -- (5,0.9);
\draw[very thick,dashed] (0.9,0) -- (5,4.1);
\draw[very thick,dotted] (1.2,0) -- (5,1.9);
\node[font=\scriptsize,anchor=south west] at (0.05,0.95) {$y=0$};
\node[font=\scriptsize,anchor=south west] at (2.9,2.35) {$x=y+1$};
\node[font=\scriptsize,anchor=south west] at (2.4,1.05) {$x=1+2y$};
\end{tikzpicture}

\smallskip
{\footnotesize Color $0$ in $(x,y)=(i+2k,k)$}
\end{minipage}
\caption{Schematic defect geometry on $P_0$. The left and right panels use the bulk coordinates of Lemma~\ref{lem:routeE-bulkcoords}; the center panel (color~$2$) uses the working frame $(x,y)=(i,\,i+k)$. In each panel the arrows indicate the generic bulk branch, while the solid/dashed/dotted lines represent the affine defect families from Lemma~\ref{lem:routeE-finite-defect}. The isolated boundary-point corrections sit on these lines and act as the splice points described in Remark~\ref{rem:routeE-design}.}
\label{fig:routeE-defects}
\end{figure}

\begin{remark}[reading guide for the working coordinates]\label{rem:coord-guide}
Colors~$1$ and~$0$ are analyzed in bulk coordinates $(u,t)$, where $t$ is the return clock and $u$ is the lane label; the transversal is $t=0$.
Color~$2$ is cleaner in the separate frame $(x,y)=(i,\,i+k)$, where the bulk move is the primitive diagonal $(x,y)\mapsto(x+1,y-1)$ and the transversal is $y=0$.
Table~\ref{tab:routeE-summary} and Figure~\ref{fig:routeE-defects} are the standing reference for these coordinate choices.
\end{remark}

\begin{remark}[$R_2$ is universal]\label{rem:R2-universal}
The return-map formula for color~$2$ is the same for all even $m\ge 6$, independent of $m\bmod 6$; only colors~$0$ and~$1$ require a Case~I\,/\,II split.
For this reason, the main text works out color~$2$ completely before the appendix turns to the genuinely mod-$6$-dependent colors.
\end{remark}

\subsection{Primary geometry, obstruction, and the Route~E repair}\label{sec:primary-obstruction}

We now isolate the conceptual role of the Case~II family.
The natural baseline model keeps the layers $S\ge 1$ of Definition~\ref{def:routeE} unchanged and uses the Case~I layer-$0$ families for all even $m$.

\begin{definition}[primary Route~E geometry]\label{def:routeE-primary}
Let $m\ge 6$ be even.
The \emph{primary Route~E geometry} is the direction assignment obtained from Definition~\ref{def:routeE} by keeping the rules on layers $S\ge 1$ unchanged and, on layer $S=0$, using the Case~I families $X_{102},X_{021},X_{210}$ irrespective of $m\bmod 6$.
Thus the primary geometry removes the Case~II repair family but preserves the same layer-$1$ and layer-$2$ transducer.
\end{definition}

\begin{proposition}[primary geometry and the first obstruction]\label{prop:primary-geometry-obstruction}
Let $m=6q+4\ge 10$, and let $T_c^{\mathrm{pri}}$ denote the induced lane maps for the primary Route~E geometry on the same transversals as in the actual Route~E proof.
Then:
\begin{enumerate}[label=\textup{(\roman*)},leftmargin=2.5em]
\item color~$2$ is unchanged from the actual Route~E construction and remains Hamiltonian;
\item color~$0$ has the same first-return formulas and the same splice order as in Case~I, hence remains Hamiltonian;
\item color~$1$ also has the Case~I first-return formula
\[
T_1^{\mathrm{pri}}(x)=
\begin{cases}
2,&x=0,\\
x+3,&1\le x\le m-4,\\
1,&x=m-3,\\
0,&x=m-2,\\
3,&x=m-1,
\end{cases}
\]
and therefore splits into the three closed cycles
\[
(0,2,5,8,\dots,m-2),
\qquad
(1,4,7,10,\dots,m-3),
\qquad
(3,6,9,12,\dots,m-1).
\]
\end{enumerate}
In particular, the first obstruction in the primary geometry occurs in color~$1$ when $m\equiv 4\pmod 6$.
\end{proposition}

\begin{proof}
For color~$2$, the Case~II repair family changes only the color-$0$ and color-$1$ letters of the layer-$0$ triples, never the color-$2$ letter.
So the color-$2$ return map is identical to the one in Proposition~\ref{prop:R2-main-data}, and the proof of Corollary~\ref{cor:R2-main} applies verbatim.

For color~$0$, the derivation of Proposition~\ref{prop:R0-caseI-data} uses only the primary defect support $t=0$ and $u=t+1$ together with the associated boundary corrections.
The congruence assumption $m\equiv 0,2\pmod 6$ enters there only because the actual Route~E construction adds the extra Case~II repair line $u=1+2t$ when $m\equiv 4\pmod 6$.
Since the primary geometry omits that extra line for all even $m$, the same first-return formulas and return-time sum remain valid, and the proof of Corollary~\ref{cor:R0-caseI} shows that color~$0$ is Hamiltonian.

For color~$1$, the same reasoning with Proposition~\ref{prop:R1-caseI} shows that the primary geometry has the displayed Case~I formula.
Here the generic rule is $x\mapsto x+3$, while the terminal values satisfy $T_1^{\mathrm{pri}}(m-3)=1$, $T_1^{\mathrm{pri}}(m-2)=0$, and $T_1^{\mathrm{pri}}(m-1)=3$.
Hence each residue-$3$ block closes on itself, giving exactly the three displayed cycles.
So the first genuine obstruction in the primary geometry is the residue-$3$ closure of color~$1$ when $m\equiv 4\pmod 6$.
\end{proof}

The actual Case~II rule adds precisely the Case~II repair family needed to alter that interior closure.
To see that this is not merely a boundary adjustment, keep the full Case~II scaffold but delete only the Case~II repair family itself.
Let
\[
Y_{210}^{\mathrm{add}}
:=
\{(1,j,m-1-j):2\le j\le m-2\}\cup\{(2,0,m-2),(2,m-1,m-1)\}
\subseteq Y_{210},
\]
and form the auxiliary direction assignment obtained from Definition~\ref{def:routeE} in Case~II by restoring the default triple $(1,2,0)$ on $Y_{210}^{\mathrm{add}}$ and leaving all other Case~II points unchanged.
Write $\widetilde T_c$ for its induced lane maps.

\begin{proposition}[deleting the Case~II repair family is already noninjective]\label{prop:boundary-no-go}
Assume $m=6q+4\ge 10$.
For the auxiliary direction assignment obtained by deleting only $Y_{210}^{\mathrm{add}}$ from the actual Case~II rule, one has
\[
\widetilde T_1(1)=\widetilde T_1(m-1)=3,
\qquad
\widetilde T_0(1)=\widetilde T_0(2)=4.
\]
In particular, both $\widetilde T_1$ and $\widetilde T_0$ are noninjective.
So within the actual Case~II scaffold, the Case~II repair family performs genuine bulk repair that the remaining boundary corrections do not supply.
\end{proposition}

\begin{proof}
For color~$1$, work in the bulk coordinates $\Phi_1(i,k)=(u,t)=(i-k,k)$ from Lemma~\ref{lem:routeE-bulkcoords}, so the generic branch is $(u,t)\mapsto(u,t+1)$.
The deleted variant removes the interior defect line $u+t=1$ but keeps the other Case~II boundary points.
Starting from lane $1$, the retained isolated point $(u,t)=(1,0)$ still carries the $+2$ stall, so
\[
(1,0)\longmapsto(3,0),
\]
and therefore $\widetilde T_1(1)=3$.
Starting from lane $m-1$, the first step is generic,
\[
(m-1,0)\xrightarrow{\bulk}(m-1,1),
\]
then the primary defect line $u+t=0$ gives the usual $+1$ stall,
\[
(m-1,1)\longmapsto(0,1),
\]
after which the orbit climbs generically to the retained boundary point $(0,m-1)$:
\[
(0,1)\xrightarrow{\bulk^{m-2}}(0,m-1).
\]
There the retained Case~II boundary correction gives a $+2$ stall,
\[
(0,m-1)\longmapsto(2,m-1),
\]
followed by the retained point corresponding to $(i,k)=(1,m-1)$, which gives a final $+1$ stall,
\[
(2,m-1)\longmapsto(3,m-1)\xrightarrow{\bulk}(3,0).
\]
Hence $\widetilde T_1(m-1)=3$.
So $\widetilde T_1$ is not injective.

For color~$0$, work in the bulk coordinates $\Phi_0(i,k)=(u,t)=(i+2k,k)$.
Here the deleted variant removes the interior Case~II line $u=1+2t$ but keeps the remaining Case~II boundary points.
Starting from lane $1$, the retained isolated point $(1,0)$ still carries the jump
\[
(1,0)\longmapsto(3,2),
\]
and the primary defect line $u=t+1$ then gives
\[
(3,2)\longmapsto(4,4).
\]
No further defect is encountered before return, so the orbit climbs generically to $(4,0)$ and $\widetilde T_0(1)=4$.
Starting from lane $2$, the ordinary base defect on $t=0$ gives
\[
(2,0)\longmapsto(3,1)\xrightarrow{\bulk}(3,2)\longmapsto(4,4),
\]
and again the orbit then returns generically to $(4,0)$.
Thus $\widetilde T_0(2)=4$ as well.
Therefore $\widetilde T_0$ is not injective.
\end{proof}

\begin{example}[the Case~II repair at $m=10$]\label{ex:m10-color1}
In the primary geometry of Proposition~\ref{prop:primary-geometry-obstruction}, the color-$1$ lane map is
\[
T_1^{\mathrm{pri}}=(0,2,5,8)(1,4,7)(3,6,9),
\]
so the primary model splits into three residue-$3$ cycles.
By Proposition~\ref{prop:R1-caseII}, the actual Route~E map at $m=10$ is the single cycle
\[
0\mapsto 2\mapsto 5\mapsto 1\mapsto 3\mapsto 9\mapsto 7\mapsto 4\mapsto 6\mapsto 8\mapsto 0.
\]
Thus the Case~II repair family does more than reconnect endpoints: it changes the interior drift on even lanes from the primary $+3$ rule to $+2$, and on odd lanes from $+3$ to $+6$, thereby interleaving the three residue-$3$ strands into one cycle.

A representative repaired orbit already appears on lane $3$.
In the color-$1$ bulk coordinates $(u,t)$, one has
\begin{multline*}
(3,0)\xrightarrow{\bulk^3}(3,3)\xrightarrow{+2}(5,3)\xrightarrow{\bulk^2}(5,5)\xrightarrow{+1}(6,5)\\
\xrightarrow{+1}(7,5)\xrightarrow{\bulk}(7,6)\xrightarrow{+2}(9,6)\xrightarrow{\bulk^4}(9,0),
\end{multline*}
so $T_1(3)=9$ and $\rho_1(3)=14=m+4$.
The two $+2$ stalls come from the diagonal defect line, while the two consecutive $+1$ stalls come from the two Case~II vertical lines $u+t=0$ and $u+t=1$.
\end{example}

This is the viewpoint to keep in mind when reading the Case~II appendix formulas: on long arithmetic families the ordered defect heights are constant, so the added lines $u+t=1$ and $u=1+2t$ change the interior shifts before Proposition~\ref{prop:routeE-splice} reduces the closure step to a finite splice permutation.

\begin{example}[Color $2$ for $m=6$]\label{ex:m6-color2}
We illustrate the return-map machinery on the smallest admissible even value $m=6$.
In the working frame $(x,y)=(i,\,i+k)$ on $P_0$, the generic (bulk) step is $G\colon(x,y)\mapsto(x+1,y-1)$, and the transversal is $L_2=\{(x,0):x\in\mathbb Z_6\}$.

Starting from lane $x=0$:
\[
(0,0)\xrightarrow{E}(1,0).
\]
This single defect step already returns to $L_2$, giving $T_2(0)=1$ and $\rho_2(0)=1$.

Starting from lane $x=1$, the orbit crosses two defect lines before returning:
\[
(1,0)\xrightarrow{G}(2,5)\xrightarrow{C}(2,4)\xrightarrow{G}(3,3)
\xrightarrow{D}(3,1)\xrightarrow{B}(5,0).
\]
The defect sequence is $C$ (line $y=m{-}1$), then $D$ (diagonal $y=x$), then $B$ (line $y=1$).
Return data: $T_2(1)=5$, $\rho_2(1)=5$.

For the remaining lanes, the same method gives $T_2(2)=0$ with $\rho_2(2)=12=2m$, and $T_2(x)=x-1$ with $\rho_2(x)=6=m$ for $x\in\{3,4,5\}$.
The first-return map $T_2$ is therefore the $6$-cycle
\[
0\to 1\to 5\to 4\to 3\to 2\to 0,
\]
and the return-time sum is $1+5+12+6+6+6=36=6^2=m^2$.
Lemma~\ref{lem:counting} now implies that $R_2$ is a single $36$-cycle on $P_0$.
\end{example}

\subsection{The color-\texorpdfstring{$2$}{2} return map}\label{sec:even-r2}
We now carry out the color-$2$ computation in full.
This is the cleanest representative case because its return map does not depend on $m\bmod 6$.
For color~$2$ it is convenient to work directly in the coordinates $(x,y)=(i,\,i+k)$ rather than the bulk frame of Lemma~\ref{lem:routeE-bulkcoords}; in this frame the generic branch becomes $G\colon(x,y)\mapsto(x+1,y-1)$ and the defect lines are those shown in the center panel of Figure~\ref{fig:routeE-defects}.

\begin{lemma}\label{lem:R2-main-xy}
In the coordinates $(x,y)=(i,i+k)$ on $P_0$, the associated Route~E $m$-step map is
\[
R_2(x,y)=
\begin{cases}
(x+1,y-2),&(x,y)=(2,2),\\
(x+2,y-1),&y=1,\ (x,y)\notin\{(1,1),(m-1,1)\},\\
(x,y-1),&\bigl(y=m-1,\ x\neq m-1\bigr)\ \text{or}\ (x,y)=(m-1,m-2),\\
(x,y-2),&\bigl(y=x,\ x\notin\{0,2,m-1\}\bigr)\ \text{or}\ (x,y)=(m-1,0),\\
(x+1,y),&\bigl(x=m-1,\ y\notin\{0,m-2\}\bigr)\ \text{or}\ (x,y)=(0,0),\\
(x+1,y-1),&\text{otherwise}.
\end{cases}
\]
All arithmetic is modulo $m$.
\end{lemma}

\begin{proof}
Substitute $i=x$ and $k=y-x$ into the closed-form formula for $R_2$ from Appendix~\ref{app:routeE} and simplify each branch.
The six displayed moves are exactly the transformed images of the six original branches.
\end{proof}

For the orbit traces below write
\[
G(x,y)=(x+1,y-1)
\]
for the generic branch of Lemma~\ref{lem:R2-main-xy}.
The other five displayed branches will be referred to as $A,B,C,D,E$ in the order listed in the lemma:
$A$ applies at the isolated point $(2,2)$;
$B$ is the defect on the line $y=1$;
$C$ is the defect on $y=m-1$;
$D$ is the defect on $y=x$;
$E$ is the defect on $x=m-1$.

The transversal is
\[
L_2=\{(x,0):x\in \Zm\}.
\]

\begin{proposition}\label{prop:R2-main-data}
For every even $m\ge 6$, the first-return map on $L_2$ is
\[
T_2(x)=
\begin{cases}
1,&x=0,\\
m-1,&x=1,\\
0,&x=2,\\
x-1,&3\le x\le m-1,
\end{cases}
\]
with return times
\[
\rho_2(x)=
\begin{cases}
1,&x=0,\\
m-1,&x=1,\\
2m,&x=2,\\
m,&3\le x\le m-1.
\end{cases}
\]
\end{proposition}

\begin{proof}
We treat the boundary values $x\in\{0,1,2,4,m-2,m-1\}$ individually and then give the generic odd-$x$ and even-$x$ formulas.
Fix $x\in \Zm$ and start from $(x,0)\in L_2$.

\smallskip
\noindent\textbf{Case $x=0$.}
The point $(0,0)$ is branch $E$, so
\[
(0,0)\xrightarrow{E}(1,0).
\]
Hence $T_2(0)=1$ and $\rho_2(0)=1$.

\smallskip
\noindent\textbf{Case $x=1$.}
The orbit is
\[
(1,0)\xrightarrow{G}(2,m-1)\xrightarrow{C}(2,m-2)
\xrightarrow{G^{\frac m2-2}}\left(\frac m2,\frac m2\right)
\xrightarrow{D}\left(\frac m2,\frac m2-2\right)
\xrightarrow{G^{\frac m2-3}}(m-3,1)\xrightarrow{B}(m-1,0).
\]
Between the displayed defect points, no other special branch can occur:
before reaching $\left(\frac m2,\frac m2\right)$ the orbit has $x<m-1$, $y\notin\{0,1,m-1\}$, and meets the diagonal only at the displayed point; after the $D$-move the next special point is $(m-3,1)$.
Therefore
\[
T_2(1)=m-1,\qquad
\rho_2(1)=1+1+\left(\frac m2-2\right)+1+\left(\frac m2-3\right)+1=m-1.
\]

\smallskip
\noindent\textbf{Case $x=2$.}
The orbit is
\begin{multline*}
(2,0)\xrightarrow{G}(3,m-1)\xrightarrow{C}(3,m-2)
\xrightarrow{G^{m-4}}(m-1,2)\xrightarrow{E}(0,2)\xrightarrow{G}(1,1) \\
\xrightarrow{D}(1,m-1)\xrightarrow{C}(1,m-2)
\xrightarrow{G^{m-3}}(m-2,1)\xrightarrow{B}(0,0).
\end{multline*}
Again, every point between the displayed defects is generic:
before $(m-1,2)$ the orbit avoids the diagonal and the line $y=1$;
after $(0,2)$ the only diagonal hit is $(1,1)$;
after $(1,m-2)$ the next defect is $(m-2,1)$.
Hence
\[
T_2(2)=0,\qquad
\rho_2(2)=1+1+(m-4)+1+1+1+1+(m-3)+1=2m.
\]

\smallskip
\noindent\textbf{Case $x$ odd, $3\le x\le m-3$.}
Set
\[
d=\frac{x+m-1}{2}.
\]
Then the orbit is
\begin{multline*}
(x,0)\xrightarrow{G}(x+1,m-1)\xrightarrow{C}(x+1,m-2)
\xrightarrow{G^{\frac{m-x-3}{2}}}(d,d)\xrightarrow{D}(d,d-2) \\
\xrightarrow{G^{\frac{m-x-1}{2}}}(m-1,x-2)\xrightarrow{E}(0,x-2)
\xrightarrow{G^{x-3}}(x-3,1)\xrightarrow{B}(x-1,0).
\end{multline*}
The first generic block runs on the line of ordinary sum $x+m-1$, which meets the diagonal exactly at $(d,d)$.
After the $D$-move the orbit runs on the line of ordinary sum $x+m-3$, and its next defect is $(m-1,x-2)$.
After the $E$-move the orbit runs on the line of ordinary sum $x-2$, whose first defect is $(x-3,1)$.
Thus
\[
T_2(x)=x-1,
\]
and
\[
\rho_2(x)=
1+1+\frac{m-x-3}{2}+1+\frac{m-x-1}{2}+1+(x-3)+1
=
m.
\]

\smallskip
\noindent\textbf{Case $x=4$.}
Here the orbit is
\[
(4,0)\xrightarrow{G}(5,m-1)\xrightarrow{C}(5,m-2)
\xrightarrow{G^{m-6}}(m-1,4)\xrightarrow{E}(0,4)
\xrightarrow{G^2}(2,2)\xrightarrow{A}(3,0).
\]
The displayed formula also covers the border case $m=6$, where the block $G^{m-6}$ is empty.
Hence
\[
T_2(4)=3,\qquad \rho_2(4)=1+1+(m-6)+1+2+1=m.
\]

\smallskip
\noindent\textbf{Case $x$ even, $6\le x\le m-4$.}
Set
\[
e=\frac{x}{2}.
\]
Then the orbit is
\begin{multline*}
(x,0)\xrightarrow{G}(x+1,m-1)\xrightarrow{C}(x+1,m-2)
\xrightarrow{G^{m-x-2}}(m-1,x)\xrightarrow{E}(0,x) \\
\xrightarrow{G^e}(e,e)\xrightarrow{D}(e,e-2)
\xrightarrow{G^{e-3}}(x-3,1)\xrightarrow{B}(x-1,0).
\end{multline*}
Before $(m-1,x)$ there is no diagonal hit because the first generic block lies on a line of odd ordinary sum.
After the $E$-move the orbit lies on the line of ordinary sum $x$, whose first diagonal point is $(e,e)$.
After the $D$-move the orbit lies on the line of ordinary sum $x-2$, whose next defect is $(x-3,1)$.
Therefore
\[
T_2(x)=x-1,
\]
and
\[
\rho_2(x)=
1+1+(m-x-2)+1+e+1+(e-3)+1
=
m.
\]

\smallskip
\noindent\textbf{Case $x=m-2$ (so necessarily $m\ge 8$).}
The orbit is
\begin{multline*}
(m-2,0)\xrightarrow{G}(m-1,m-1)\xrightarrow{E}(0,m-1)\xrightarrow{C}(0,m-2) \\
\xrightarrow{G^{\frac{m-2}{2}}}\left(\frac{m-2}{2},\frac{m-2}{2}\right)
\xrightarrow{D}\left(\frac{m-2}{2},\frac{m-6}{2}\right)
\xrightarrow{G^{\frac{m-8}{2}}}(m-5,1)\xrightarrow{B}(m-3,0).
\end{multline*}
Hence
\[
T_2(m-2)=m-3,\qquad
\rho_2(m-2)=1+1+1+\frac{m-2}{2}+1+\frac{m-8}{2}+1=m.
\]

\smallskip
\noindent\textbf{Case $x=m-1$.}
The orbit is
\begin{multline*}
(m-1,0)\xrightarrow{D}(m-1,m-2)\xrightarrow{C}(m-1,m-3)\xrightarrow{E}(0,m-3)\\
\xrightarrow{G^{m-4}}(m-4,1)\xrightarrow{B}(m-2,0).
\end{multline*}
So
\[
T_2(m-1)=m-2,\qquad
\rho_2(m-1)=1+1+1+(m-4)+1=m.
\]

The seven cases exhaust $x\in \Zm$, the combined return times sum to $m^2$, and the displayed formulas follow.
\end{proof}

\begin{corollary}[color $2$ is Hamilton for Route~E]\label{cor:R2-main}
For every even $m\ge 6$, the associated map $R_2$ is a single $m^2$-cycle on $P_0$.
\end{corollary}

\begin{proof}
By Proposition~\ref{prop:R2-main-data}, the induced lane map follows the cyclic order
\[
(0,1\mid m-1,m-2,\dots,2),
\]
namely $T_2(0)=1$, $T_2(1)=m-1$, $T_2(2)=0$, and $T_2(x)=x-1$ for $3\le x\le m-1$.
Therefore $T_2$ is a single $m$-cycle.
Also
\[
\sum_{x\in \Zm}\rho_2(x)=1+(m-1)+2m+(m-3)m=m^2.
\]
Lemma~\ref{lem:counting} now implies that $R_2$ is a single $m^2$-cycle.
\end{proof}

\begin{theorem}[Route~E for even $m\ge 6$]\label{thm:routeE-even-main}
For every even integer $m\ge 6$, the Route~E direction assignment (Definition~\ref{def:routeE}) is a valid coloring of $D_3(m)$, and its three color classes are directed Hamilton cycles.
\end{theorem}

\begin{proof}
The even proof may now be read as four checkpoints.
\begin{enumerate}[label=\textup{(\roman*)},leftmargin=2.5em]
\item \textbf{Why the Case~II repair family is needed.}
Proposition~\ref{prop:primary-geometry-obstruction} identifies the first obstruction in the primary geometry, and Proposition~\ref{prop:boundary-no-go} shows that within the actual Case~II scaffold the remaining boundary corrections do not remove it; the Case~II repair family is genuine bulk repair.

\item \textbf{From Definition~\ref{def:routeE} to explicit return maps.}
Proposition~\ref{prop:routeE-transducer} reduces the derivation of the repaired return maps to a finite low-layer table, and Appendix~\ref{app:routeE-derivation} carries out that evaluation, yielding the closed-form formulas for $R_0,R_1,R_2$ recorded in Appendix~\ref{app:routeE}.
Lemmas~\ref{lem:routeE-bulkcoords}, \ref{lem:routeE-finite-defect}, and \ref{lem:routeE-itinerary} show that, in adapted working coordinates, each return map is governed by a generic bulk branch together with finitely many affine defect components.

\item \textbf{Cycle closure on $P_0$.}
The main text handles color~$2$: Proposition~\ref{prop:R2-main-data} derives the first-return formulas, and Corollary~\ref{cor:R2-main} shows that $R_2$ is a single $m^2$-cycle on $P_0$.
For colors $1$ and $0$, Appendix~\ref{app:routeE-cycle} derives the repaired first-return formulas.
Propositions~\ref{prop:R1-caseI}, \ref{prop:R1-caseII}, \ref{prop:R0-caseI-data}, and \ref{prop:R0-caseII-data} give the lane maps $T_c$ and return times $\rho_c$; in Case~II these formulas encode the altered interior shifts created by the Case~II repair family.
Proposition~\ref{prop:routeE-splice} packages the remaining closure step as a finite splice permutation on arithmetic family-blocks, i.e.\ at the finitely many places where the lane dynamics deviate from the pure odometer carry.
In every congruence class that splice permutation is a single cycle, so each $T_c$ is a single $m$-cycle.
Corollaries~\ref{cor:R1-all} and~\ref{cor:R0-all}, together with Lemma~\ref{lem:counting}, then show that $R_1$ and $R_0$ are single $m^2$-cycles.

\item \textbf{Lift from $P_0$ to $V$.}
At this point every $R_c$ is bijective, so Lemma~\ref{lem:return-validity} implies that the Route~E direction assignment is a valid coloring.
Lemma~\ref{lem:return} then lifts the $m$-step return-map information to the full vertex set and shows that each color map is a directed Hamilton cycle on $V$.
\end{enumerate}
Because every Route~E direction triple is a permutation of $(0,1,2)$, the three Hamilton cycles are arc-disjoint and together cover all arcs of $D_3(m)$.
\end{proof}

\section{The finite case \texorpdfstring{$m=4$}{m=4}}\label{sec:m4}

Route~E is designed for $m\ge 6$; the case $m=4$ is handled by a separate finite witness.
Appendix~\ref{app:m4} prints the complete direction table, from which the three Hamilton cycles can be read off by finite iteration; a machine-readable supplement records the orbits explicitly.

\begin{proposition}[The finite case $m=4$]\label{prop:m4-main}
The graph $D_3(4)$ admits a decomposition into three directed Hamilton cycles.
\end{proposition}

\begin{proof}
Appendix~\ref{app:m4} records a complete direction table for a coloring of $D_3(4)$.
Starting from that table, each color orbit is obtained by finite iteration from any chosen initial vertex.
The three orbits starting from $(0,0,0)$ are recorded in the accompanying machine-readable supplement; each has length $64$ and agrees with the displayed table.
Thus the finite verification splits into two transparent steps: the table shows that every vertex sees a permutation of the three allowed directions, hence defines a valid direction assignment, and the extracted orbit data show that each color map is a single $64$-cycle on $V$, hence a bijection; together these establish a valid coloring.
Therefore the three color classes are Hamilton cycles and form a Hamilton decomposition of $D_3(4)$.
\end{proof}

\section{Completion for all \texorpdfstring{$m\ge 3$}{m >= 3}}\label{sec:completion}

\begin{proof}[Proof of Theorem~\ref{thm:main-all}]
If $m$ is odd, apply Theorem~\ref{thm:odd-hamilton}.
If $m=4$, apply Proposition~\ref{prop:m4-main}.
If $m$ is even and $m\ge 6$, apply Theorem~\ref{thm:routeE-even-main}.
These cases exhaust all integers $m\ge 3$.
\end{proof}

\section{Discussion and outlook}\label{sec:discussion}

The paper is best read as a theorem about an odometer hidden inside the directed torus.
The layer function provides the ambient clock, and the proof identifies the relevant clock-and-carry dynamics on the section $P_0$.
In the odd case, this yields an explicit affine conjugacy to the standard odometer.
In the even case, the parity barrier rules out the canonical Kempe route, so the construction rebuilds the mechanism directly from a low-layer repair and a further first-return reduction to explicit lanes.
The boundary case $m=4$ is the finite exceptional instance, recorded by a table rather than by a uniform arithmetic family.
The accompanying supplementary files are independent regression checks and do not form part of the proof.

Several structural questions remain.
First, can the even-case analysis be compressed into a Route~E-specific critical-lane theorem that reads off the repaired arithmetic family-blocks and the splice permutation directly from the affine defect heights, bypassing most of the remaining itinerary bookkeeping?
Relatedly, can the color-$0$ $m\bmod 12$ split be derived uniformly as a cyclic reordering inside a single residue-$4$ splice picture?
Second, the proof suggests a higher-dimensional program in which low-layer defect arrangements are analyzed by nested return sections until a finite-defect odometer emerges.
The sign-product invariant already points to genuine parity obstructions for Kempe-from-canonical strategies, while the Route~E construction suggests that obstruction-breaking witnesses should still be organized by clocks, carries, and finitely many splices.
The first natural test case is the directed $4$-torus, where one can ask whether the same return-section viewpoint still isolates a finite-defect clock-and-carry core.
More broadly, it is natural to ask whether composite higher-dimensional cases admit comparable nested return maps and finite-defect odometer models.
It would be especially interesting to understand whether higher-dimensional directed tori admit equally clean nested-return descriptions and whether odd and even ambient dimensions exhibit genuinely different parity behavior.

\paragraph{Software and formalization.}
Verification scripts and machine-readable witness data accompany the arXiv submission as ancillary files.
A Lean~4 formalization covering the odd-$m$ affine conjugacy argument and the even-$m$ Route~E construction is maintained at
\url{https://github.com/aria1th/Torus-Hamilton-Decomposition/}.
The entire formalization compiles against current Mathlib (\texttt{lake build TorusD3Odometer}).
Large language models assisted with writing and coding; all mathematical content was verified by the author.

\appendix

\paragraph{How to read the appendices.}
Appendix~\ref{app:routeE} is the lookup appendix: it records the closed formulas for the Route~E return maps $R_0,R_1,R_2$.
Appendix~\ref{app:routeE-derivation} is the derivation appendix: starting from Definition~\ref{def:routeE}, it evaluates the low-layer transducer and proves Proposition~\ref{prop:routeE-return-formulas}, which is the source of the displayed formulas.
Appendix~\ref{app:routeE-cycle} is the cycle-structure appendix: taking those formulas as input, it derives the first-return maps and return-time sums for colors~$1$ and~$0$ and packages the last closure step in Proposition~\ref{prop:routeE-splice}.
A reader mainly interested in proof flow may therefore read the appendices in the order Appendix~\ref{app:routeE}, then Appendix~\ref{app:routeE-cycle}, and finally Appendix~\ref{app:routeE-derivation}; Table~\ref{tab:appendix-cycle-guide} gives a color-by-color guide to the Hamiltonicity appendix.
Appendix~\ref{app:m4} and Appendix~\ref{app:verification} are archival supplements for the finite witness and the independent regression checks.
\section{Route~E return maps}\label{app:routeE}

Definition~\ref{def:routeE} appears in Section~\ref{sec:even}.
This appendix records the closed-form return maps on $P_0$ derived from that definition in Appendix~\ref{app:routeE-derivation} and analyzed in Appendix~\ref{app:routeE-cycle}.

\paragraph{Return maps on $P_0$.}
Parameterize $P_0=\{S=0\}$ by
\[
v(i,k):=(i,-i-k,k),\qquad (i,k)\in\mathbb Z_m^2.
\]
For the explicit Route~E direction assignment, define $R_c=f_c^m|_{P_0}$.
(We write $R_c$ rather than $F_c$ to distinguish the associated $m$-step maps of Route~E from the return maps $F_c$ of the Kempe-swap coloring analyzed in Section~\ref{sec:affine}.)
The following piecewise formulas are derived directly from Definition~\ref{def:routeE} in Appendix~\ref{app:routeE-derivation}.

\medskip
\noindent\textbf{Case I: $m\equiv 0,2\pmod 6$.}
\[
R_0(i,k)=
\begin{cases}
(i-2,\ k) & (i,k)=(0,0),\\
(i-3,\ k+3) & (i,k)=(1,m-1),\\
(i-2,\ k+2) & (i,k)=(m-2,0),\\
(i-1,\ k) & (i,k)=(m-1,1),\\
(i-3,\ k+2) & i+k\equiv 1\ \text{and}\ (i,k)\neq (1,0),\\
(i-1,\ k+1) & ((k=0\ \text{and}\ 1\le i\le m-3)\ \text{or}\ (i,k)\in\{(0,m-1),(m-2,1)\}),\\
(i-2,\ k+1) & \text{otherwise.}
\end{cases}
\]
\[
R_1(i,k)=
\begin{cases}
(i+2,\ k) & \begin{aligned}[t]
            &\bigl( (i+k\equiv m-1\ \text{and}\ 1\le i\le m-3)\\
            &\quad\text{or}\ (i,k)\in\{(0,0),(m-2,0),(m-1,m-1)\} \bigr),
            \end{aligned}\\[2.5ex]
(i+1,\ k) & \begin{aligned}[t]
            &\bigl( (i=0\ \text{and}\ 1\le k\le m-2)\\
            &\quad\text{or}\ (i,k)\in\{(1,m-1),(m-2,1)\} \bigr),
            \end{aligned}\\[2.5ex]
(i+1,\ k+1) & \text{otherwise.}
\end{cases}
\]

\medskip
\noindent\textbf{Case II: $m\equiv 4\pmod 6$.}
\[
R_0(i,k)=
\begin{cases}
(i-2,\ k) & (i,k)=(0,0),\\
(i-1,\ k) & (i,k)=(m-1,1),\\
(i-3,\ k+3) & (i,k)\in\{(1,m-1),(2,m-2)\},\\
(i-2,\ k+2) & \bigl( (i=1\ \text{and}\ 0\le k\le m-3)\ \text{or}\ (i,k)\in\{(2,m-1),(m-2,0)\} \bigr),\\
(i-3,\ k+2) & i+k\equiv 1\ \text{and}\ (i,k)\notin\{(1,0),(2,m-1)\},\\
(i-1,\ k+1) & \begin{aligned}[t]
              &\bigl( (k=0\ \text{and}\ 2\le i\le m-3)\\
              &\quad\text{or}\ (i,k)\in\{(0,m-1),(1,m-2),(m-2,1)\} \bigr),
              \end{aligned}\\[2.5ex]
(i-2,\ k+1) & \text{otherwise.}
\end{cases}
\]
\[
R_1(i,k)=
\begin{cases}
(i+2,\ k) & \begin{aligned}[t]
            &\bigl( (i+k\equiv m-1\ \text{and}\ 2\le i\le m-3)\\
            &\quad\text{or}\ (i,k)\in\{(0,0),(1,0),(m-2,0),(m-1,m-1)\} \bigr),
            \end{aligned}\\[2.5ex]
(i+1,\ k) & \begin{aligned}[t]
            &\bigl( (i=0\ \text{and}\ 1\le k\le m-2)\ \text{or}\ (i=1\ \text{and}\ 1\le k\le m-1)\\
            &\quad\text{or}\ (i,k)\in\{(2,m-2),(2,m-1),(m-2,1)\} \bigr),
            \end{aligned}\\[2.5ex]
(i+1,\ k+1) & \text{otherwise.}
\end{cases}
\]

\medskip
\noindent\textbf{Uniform for all even $m\ge 6$.}
\[
R_2(i,k)=
\begin{cases}
(i+1,\ k-3) & (i,k)=(2,0),\\
(i+2,\ k-3) & i+k\equiv 1\ \text{and}\ (i,k)\notin\{(1,0),(m-1,2)\},\\
(i,\ k-1) & ((k\equiv m-1-i\ \text{and}\ i\not\equiv m-1)\ \text{or}\ (i,k)=(m-1,m-1)),\\
(i,\ k-2) & ((k=0\ \text{and}\ i\notin\{0,2,m-1\})\ \text{or}\ (i,k)=(m-1,1)),\\
(i+1,\ k-1) & ((i=m-1\ \text{and}\ k\notin\{1,m-1\})\ \text{or}\ (i,k)=(0,0)),\\
(i+1,\ k-2) & \text{otherwise.}
\end{cases}
\]

\section{Derivation of the Route~E return maps}\label{app:routeE-derivation}

In this appendix we derive the return maps recorded in Appendix~\ref{app:routeE} directly from the low-layer Route~E rule of Definition~\ref{def:routeE}.
The argument does not appeal to the defect-cycle structure; it tracks only the first three layers $S=0,1,2$.

\begin{proposition}\label{prop:routeE-return-formulas}
For every even $m\ge 6$, the return maps $R_0,R_1,R_2$ induced by Definition~\ref{def:routeE} are exactly the piecewise formulas displayed in Appendix~\ref{app:routeE}.
\end{proposition}

\begin{proof}
Parameterize
\[
P_0=\{(i,j,k)\in \mathbb Z_m^3:i+j+k\equiv 0\}
\]
by
\[
v(i,k):=(i,-i-k,k),\qquad (i,k)\in \mathbb Z_m^2.
\]
Fix a color $c\in\{0,1,2\}$ and start from $v(i,k)\in P_0$.
Write the first three color-$c$ directions encountered on the layers $S=0,1,2$ as
\[
\omega_c(i,k)=(a_c,b_c,c_c)\in\{0,1,2\}^3.
\]
After those three steps the orbit lies on layer $S=3$.
From there until the next return to $P_0$ all layers are canonical, so the remaining $m-3$ steps are all in direction $c$.
Hence the whole return map is determined by the low-layer word $\omega_c(i,k)$.

If $N_r(\omega)$ denotes the number of occurrences of direction $r$ in a word $\omega$, then
\[
R_0(i,k)=\bigl(i+N_0(\omega_0)-3,\ k+N_2(\omega_0)\bigr),
\]
\[
R_1(i,k)=\bigl(i+N_0(\omega_1),\ k+N_2(\omega_1)\bigr),
\]
\[
R_2(i,k)=\bigl(i+N_0(\omega_2),\ k+N_2(\omega_2)-3\bigr).
\]
So it remains only to determine the low-layer words.

From Definition~\ref{def:routeE}:
\begin{itemize}[leftmargin=2em]
\item on layer $S=1$, the triple is $(1,0,2)$ if $i=0$, and $(2,0,1)$ if $i\ne 0$;
\item on layer $S=2$, the triple is $(2,1,0)$ if $j=0$, and $(0,1,2)$ if $j\ne 0$.
\end{itemize}
Therefore:
\begin{itemize}[leftmargin=2em]
\item for color $0$, the layer-$1$ direction depends on whether the post-step-$1$ $i$-coordinate equals $0$, and the layer-$2$ direction depends on whether the post-step-$2$ $j$-coordinate equals $0$;
\item for color $1$, the directions on layers $1$ and $2$ are always $0$ and $1$;
\item for color $2$, the layer-$1$ direction depends on whether the post-step-$1$ $i$-coordinate equals $0$, and the layer-$2$ direction depends on whether the post-step-$2$ $j$-coordinate equals $0$.
\end{itemize}

\paragraph{Case I: $m\equiv 0,2\pmod 6$.}
In the $(i,k)$-coordinates on $P_0$, the layer-$0$ direction triples become
\begin{align*}
A&:=\{(0,0)\}\cup\{(i,k):i+k\equiv m-1,\ 1\le i\le m-3\}\cup\{(m-1,m-1)\},\\
B&:=\{(0,m-1)\}\cup\{(i,0):1\le i\le m-3\}\cup\{(m-1,1)\},\\
C&:=\{(0,k):1\le k\le m-2\}\cup\{(1,m-1)\},\\
D&:=\{(m-2,1)\},\qquad E:=\{(m-2,0)\},
\end{align*}
corresponding respectively to the triples $102,021,210,012,201$, while the remaining points
\[
F:=\mathbb Z_m^2\setminus (A\cup B\cup C\cup D\cup E)
\]
carry the triple $120$.

\paragraph{Color $1$.}
On layers $1$ and $2$, color $1$ always uses directions $0$ and $1$.
So the low-layer word is always of the form $(a,0,1)$, where $a$ is the layer-$0$ color-$1$ direction.
Thus:
\begin{itemize}[leftmargin=2em]
\item if $a=0$, then $\omega_1=001$ and $R_1(i,k)=(i+2,k)$;
\item if $a=1$, then $\omega_1=101$ and $R_1(i,k)=(i+1,k)$;
\item if $a=2$, then $\omega_1=201$ and $R_1(i,k)=(i+1,k+1)$.
\end{itemize}
Reading off the sets on which the layer-$0$ color-$1$ direction equals $0$, $1$, or $2$ gives exactly the Case~I formula for $R_1$ in Appendix~\ref{app:routeE}.

\paragraph{Color $0$.}
For color $0$, we determine the word $\omega_0=(a_0,b_0,c_0)$ algebraically.
Step $1$ goes in direction $a_0$, producing $i_1=i+\mathbf{1}_{a_0=0}$.
The layer-$1$ rule gives $b_0=1$ if $i_1\equiv 0 \pmod m$, and $2$ otherwise.
Step $2$ goes in direction $b_0$, producing
\[
j_2=j+\mathbf{1}_{a_0=1}+\mathbf{1}_{b_0=1},\qquad j\equiv -i-k \pmod m.
\]
The layer-$2$ rule gives $c_0=2$ if $j_2\equiv 0 \pmod m$, and $0$ otherwise.
Table~\ref{tab:c0_case1} executes this arithmetic trace on each region of $P_0$.
Grouping the start points by the final displacement yields exactly the displayed Case~I formula for $R_0$.

\begin{table}[H]
\centering
\scriptsize
\setlength{\tabcolsep}{4pt}
\renewcommand{\arraystretch}{1.15}
\caption{Derivation of low-layer words $\omega_0$ for color $0$, Case~I ($m\equiv 0,2\pmod 6$).}
\label{tab:c0_case1}
\resizebox{\textwidth}{!}{%
\begin{tabular}{@{}llccccccl@{}}
\toprule
Set & Condition / point & $a_0$ & $i_1=i+\mathbf{1}_{a_0=0}$ & $b_0$ & $j_2=j+\mathbf{1}_{a_0=1}+\mathbf{1}_{b_0=1}$ & $c_0$ & $\omega_0$ & Disp. $R_0(i,k)$ \\
\midrule
$A$ & $(0,0)$ & $1$ & $0\Rightarrow$ & $1$ & $0+1+1=2\not\equiv 0\Rightarrow$ & $0$ & $110$ & $(i-2,k)$ \\
& $A\setminus\{(0,0)\}$ & $1$ & $i\not\equiv 0\Rightarrow$ & $2$ & $j+1+0\not\equiv 0\Rightarrow$ & $0$ & $120$ & $(i-2,k+1)$ \\
\midrule
$B$ & $(m-1,1)$ & $0$ & $0\Rightarrow$ & $1$ & $0+0+1=1\not\equiv 0\Rightarrow$ & $0$ & $010$ & $(i-1,k)$ \\
& $B\setminus\{(m-1,1)\}$ & $0$ & $i+1\not\equiv 0\Rightarrow$ & $2$ & $j+0+0=j\not\equiv 0\Rightarrow$ & $0$ & $020$ & $(i-1,k+1)$ \\
\midrule
$C$ & $(0,1)$ & $2$ & $0\Rightarrow$ & $1$ & $-1+0+1=0\Rightarrow$ & $2$ & $212$ & $(i-3,k+2)$ \\
& $(1,m-1)$ & $2$ & $1\not\equiv 0\Rightarrow$ & $2$ & $0+0+0=0\Rightarrow$ & $2$ & $222$ & $(i-3,k+3)$ \\
& remaining $C$ & $2$ & $0\Rightarrow$ & $1$ & $j+0+1\not\equiv 0\Rightarrow$ & $0$ & $210$ & $(i-2,k+1)$ \\
\midrule
$D$ & $(m-2,1)$ & $0$ & $m-1\not\equiv 0\Rightarrow$ & $2$ & $1+0+0=1\not\equiv 0\Rightarrow$ & $0$ & $020$ & $(i-1,k+1)$ \\
\midrule
$E$ & $(m-2,0)$ & $2$ & $m-2\not\equiv 0\Rightarrow$ & $2$ & $2+0+0=2\not\equiv 0\Rightarrow$ & $0$ & $220$ & $(i-2,k+2)$ \\
\midrule
$F$ & $i+k\equiv 1$ (equiv. $j\equiv -1$) & $1$ & $i\not\equiv 0\Rightarrow$ & $2$ & $-1+1+0=0\Rightarrow$ & $2$ & $122$ & $(i-3,k+2)$ \\
& otherwise & $1$ & $i\not\equiv 0\Rightarrow$ & $2$ & $j+1+0\not\equiv 0\Rightarrow$ & $0$ & $120$ & $(i-2,k+1)$ \\
\bottomrule
\end{tabular}}
\end{table}

\paragraph{Color $2$.}
For color $2$, the layer-$1$ test is again whether the post-step-$1$ $i$-coordinate vanishes, while the layer-$2$ test is whether the post-step-$2$ $j$-coordinate vanishes.
Table~\ref{tab:c2_case1} evaluates these tests on every region of the Case~I partition.
Grouping the resulting low-layer words by displacement yields exactly the uniform formula for $R_2$ displayed in Appendix~\ref{app:routeE}.

\begin{table}[H]
\centering
\scriptsize
\setlength{\tabcolsep}{4pt}
\renewcommand{\arraystretch}{1.15}
\caption{Derivation of low-layer words $\omega_2$ for color $2$, Case~I ($m\equiv 0,2\pmod 6$).}
\label{tab:c2_case1}
\resizebox{\textwidth}{!}{%
\begin{tabular}{@{}llccccccl@{}}
\toprule
Set & Condition / point & $a_2$ & $i_1=i+\mathbf{1}_{a_2=0}$ & $b_2$ & $j_2=j+\mathbf{1}_{a_2=1}+\mathbf{1}_{b_2=1}$ & $c_2$ & $\omega_2$ & Disp. $R_2(i,k)$ \\
\midrule
$A$ & $(0,0)$ & $2$ & $0\Rightarrow$ & $2$ & $0+0+0=0\Rightarrow$ & $0$ & $220$ & $(i+1,k-1)$ \\
& $A\setminus\{(0,0)\}$ & $2$ & $i\not\equiv 0\Rightarrow$ & $1$ & $j+0+1\not\equiv 0\Rightarrow$ & $2$ & $212$ & $(i,k-1)$ \\
\midrule
$B$ & $(0,m-1)$ & $1$ & $0\Rightarrow$ & $2$ & $1+1+0=2\not\equiv 0\Rightarrow$ & $2$ & $122$ & $(i,k-1)$ \\
& $(2,0)$ & $1$ & $2\not\equiv 0\Rightarrow$ & $1$ & $(m-2)+1+1=0\Rightarrow$ & $0$ & $110$ & $(i+1,k-3)$ \\
& $B\setminus\{(0,m-1),(2,0)\}$ & $1$ & $i\not\equiv 0\Rightarrow$ & $1$ & $j+1+1\not\equiv 0\Rightarrow$ & $2$ & $112$ & $(i,k-2)$ \\
\midrule
$C$ & $(0,1)$ & $0$ & $1\not\equiv 0\Rightarrow$ & $1$ & $(m-1)+0+1=0\Rightarrow$ & $0$ & $010$ & $(i+2,k-3)$ \\
& $C\setminus\{(0,1)\}$ & $0$ & $i+1\not\equiv 0\Rightarrow$ & $1$ & $j+0+1\not\equiv 0\Rightarrow$ & $2$ & $012$ & $(i+1,k-2)$ \\
\midrule
$D$ & $(m-2,1)$ & $2$ & $m-2\not\equiv 0\Rightarrow$ & $1$ & $1+0+1=2\not\equiv 0\Rightarrow$ & $2$ & $212$ & $(i,k-1)$ \\
\midrule
$E$ & $(m-2,0)$ & $1$ & $m-2\not\equiv 0\Rightarrow$ & $1$ & $2+1+1\not\equiv 0\Rightarrow$ & $2$ & $112$ & $(i,k-2)$ \\
\midrule
$F$ & $i=m-1$ and $k\notin\{1,m-1\}$ & $0$ & $0\Rightarrow$ & $2$ & $j+0+0\not\equiv 0\Rightarrow$ & $2$ & $022$ & $(i+1,k-1)$ \\
& $i+k\equiv 1$ and $i\neq m-1$ & $0$ & $i+1\not\equiv 0\Rightarrow$ & $1$ & $-1+0+1=0\Rightarrow$ & $0$ & $010$ & $(i+2,k-3)$ \\
& otherwise & $0$ & $i+1\not\equiv 0\Rightarrow$ & $1$ & $j+0+1\not\equiv 0\Rightarrow$ & $2$ & $012$ & $(i+1,k-2)$ \\
\bottomrule
\end{tabular}}
\end{table}

\paragraph{Case II: $m\equiv 4\pmod 6$.}
In the $(i,k)$-coordinates on $P_0$, the layer-$0$ direction triples become
\begin{align*}
A&:=\{(0,0)\}\cup\{(i,k):i+k\equiv m-1,\ 2\le i\le m-3\}\cup\{(m-1,m-1)\},\\
B&:=\{(0,m-1)\}\cup\{(i,0):2\le i\le m-3\}\cup\{(m-1,1)\},\\
C&:=\{(0,k):1\le k\le m-2\}\cup\{(1,k):1\le k\le m-3\}\cup\{(1,m-1)\}\cup\{(2,m-2),(2,m-1)\},\\
D&:=\{(1,m-2),(m-2,1)\},\qquad E:=\{(1,0),(m-2,0)\},
\end{align*}
corresponding again to the triples $102,021,210,012,201$, with $120$ on the remaining points
\[
F:=\mathbb Z_m^2\setminus (A\cup B\cup C\cup D\cup E).
\]

\paragraph{Color $1$.}
Exactly the same low-layer argument applies.
In Case~II the layer-$0$ color-$1$ direction equals $0$ on
\[
(i+k\equiv m-1,\ 2\le i\le m-3)\ \text{or}\ (i,k)\in\{(0,0),(1,0),(m-2,0),(m-1,m-1)\},
\]
it equals $1$ on
\[
(i=0,\ 1\le k\le m-2)\ \text{or}\ (i=1,\ 1\le k\le m-1)\ \text{or}\ (i,k)\in\{(2,m-2),(2,m-1),(m-2,1)\},
\]
and it equals $2$ otherwise.
This is exactly the Case~II formula for $R_1$.

\paragraph{Color $0$.}
Table~\ref{tab:c0_case2} executes the same arithmetic trace for the enlarged Case~II partition.
The additional $Y_{210}$ points contribute exactly the new rows carrying the words $220$ and $222$; after grouping by displacement, one recovers the displayed Case~II formula for $R_0$.

\begin{table}[H]
\centering
\scriptsize
\setlength{\tabcolsep}{4pt}
\renewcommand{\arraystretch}{1.15}
\caption{Derivation of low-layer words $\omega_0$ for color $0$, Case~II ($m\equiv 4\pmod 6$).}
\label{tab:c0_case2}
\resizebox{\textwidth}{!}{%
\begin{tabular}{@{}llccccccl@{}}
\toprule
Set & Condition / point & $a_0$ & $i_1=i+\mathbf{1}_{a_0=0}$ & $b_0$ & $j_2=j+\mathbf{1}_{a_0=1}+\mathbf{1}_{b_0=1}$ & $c_0$ & $\omega_0$ & Disp. $R_0(i,k)$ \\
\midrule
$A$ & $(0,0)$ & $1$ & $0\Rightarrow$ & $1$ & $0+1+1=2\not\equiv 0\Rightarrow$ & $0$ & $110$ & $(i-2,k)$ \\
& $A\setminus\{(0,0)\}$ & $1$ & $i\not\equiv 0\Rightarrow$ & $2$ & $j+1+0\not\equiv 0\Rightarrow$ & $0$ & $120$ & $(i-2,k+1)$ \\
\midrule
$B$ & $(m-1,1)$ & $0$ & $0\Rightarrow$ & $1$ & $0+0+1=1\not\equiv 0\Rightarrow$ & $0$ & $010$ & $(i-1,k)$ \\
& $B\setminus\{(m-1,1)\}$ & $0$ & $i+1\not\equiv 0\Rightarrow$ & $2$ & $j+0+0\not\equiv 0\Rightarrow$ & $0$ & $020$ & $(i-1,k+1)$ \\
\midrule
$C$ & $(0,1)$ & $2$ & $0\Rightarrow$ & $1$ & $(m-1)+0+1=0\Rightarrow$ & $2$ & $212$ & $(i-3,k+2)$ \\
& $(0,k)$ with $2\le k\le m-2$ & $2$ & $0\Rightarrow$ & $1$ & $j+0+1\not\equiv 0\Rightarrow$ & $0$ & $210$ & $(i-2,k+1)$ \\
& $(1,k)$ with $1\le k\le m-2$, or $(i,k)=(2,m-1)$ & $2$ & $i\not\equiv 0\Rightarrow$ & $2$ & $j+0+0\not\equiv 0\Rightarrow$ & $0$ & $220$ & $(i-2,k+2)$ \\
& $(1,m-1)$ or $(2,m-2)$ & $2$ & $i\not\equiv 0\Rightarrow$ & $2$ & $0+0+0=0\Rightarrow$ & $2$ & $222$ & $(i-3,k+3)$ \\
\midrule
$D$ & $(1,m-2)$ or $(m-2,1)$ & $0$ & $i+1\not\equiv 0\Rightarrow$ & $2$ & $j+0+0\not\equiv 0\Rightarrow$ & $0$ & $020$ & $(i-1,k+1)$ \\
\midrule
$E$ & $(1,0)$ or $(m-2,0)$ & $2$ & $i\not\equiv 0\Rightarrow$ & $2$ & $j+0+0\not\equiv 0\Rightarrow$ & $0$ & $220$ & $(i-2,k+2)$ \\
\midrule
$F$ & $i+k\equiv 1$ (equiv. $j\equiv -1$) & $1$ & $i\not\equiv 0\Rightarrow$ & $2$ & $-1+1+0=0\Rightarrow$ & $2$ & $122$ & $(i-3,k+2)$ \\
& otherwise & $1$ & $i\not\equiv 0\Rightarrow$ & $2$ & $j+1+0\not\equiv 0\Rightarrow$ & $0$ & $120$ & $(i-2,k+1)$ \\
\bottomrule
\end{tabular}}
\end{table}

\paragraph{Color $2$.}
Table~\ref{tab:c2_case2} records the corresponding low-layer trace for color $2$ over the Case~II partition.
The only genuinely new row relative to Case~I is the split inside $C$, caused by the extra points of the Case~II repair family; after grouping by displacement one again obtains the uniform formula for $R_2$.

\begin{table}[H]
\centering
\scriptsize
\setlength{\tabcolsep}{4pt}
\renewcommand{\arraystretch}{1.15}
\caption{Derivation of low-layer words $\omega_2$ for color $2$, Case~II ($m\equiv 4\pmod 6$).}
\label{tab:c2_case2}
\resizebox{\textwidth}{!}{%
\begin{tabular}{@{}llccccccl@{}}
\toprule
Set & Condition / point & $a_2$ & $i_1=i+\mathbf{1}_{a_2=0}$ & $b_2$ & $j_2=j+\mathbf{1}_{a_2=1}+\mathbf{1}_{b_2=1}$ & $c_2$ & $\omega_2$ & Disp. $R_2(i,k)$ \\
\midrule
$A$ & $(0,0)$ & $2$ & $0\Rightarrow$ & $2$ & $0+0+0=0\Rightarrow$ & $0$ & $220$ & $(i+1,k-1)$ \\
& $A\setminus\{(0,0)\}$ & $2$ & $i\not\equiv 0\Rightarrow$ & $1$ & $j+0+1\not\equiv 0\Rightarrow$ & $2$ & $212$ & $(i,k-1)$ \\
\midrule
$B$ & $(0,m-1)$ & $1$ & $0\Rightarrow$ & $2$ & $1+1+0=2\not\equiv 0\Rightarrow$ & $2$ & $122$ & $(i,k-1)$ \\
& $(2,0)$ & $1$ & $2\not\equiv 0\Rightarrow$ & $1$ & $(m-2)+1+1=0\Rightarrow$ & $0$ & $110$ & $(i+1,k-3)$ \\
& $B\setminus\{(0,m-1),(2,0)\}$ & $1$ & $i\not\equiv 0\Rightarrow$ & $1$ & $j+1+1\not\equiv 0\Rightarrow$ & $2$ & $112$ & $(i,k-2)$ \\
\midrule
$C$ & $(0,1)$ or $(2,m-1)$ & $0$ & $i+1\not\equiv 0\Rightarrow$ & $1$ & $(m-1)+0+1=0\Rightarrow$ & $0$ & $010$ & $(i+2,k-3)$ \\
& remaining $C$ & $0$ & $i+1\not\equiv 0\Rightarrow$ & $1$ & $j+0+1\not\equiv 0\Rightarrow$ & $2$ & $012$ & $(i+1,k-2)$ \\
\midrule
$D$ & $(1,m-2)$ or $(m-2,1)$ & $2$ & $i\not\equiv 0\Rightarrow$ & $1$ & $j+0+1\not\equiv 0\Rightarrow$ & $2$ & $212$ & $(i,k-1)$ \\
\midrule
$E$ & $(1,0)$ or $(m-2,0)$ & $1$ & $i\not\equiv 0\Rightarrow$ & $1$ & $j+1+1\not\equiv 0\Rightarrow$ & $2$ & $112$ & $(i,k-2)$ \\
\midrule
$F$ & $i=m-1$ and $k\notin\{1,m-1\}$ & $0$ & $0\Rightarrow$ & $2$ & $j+0+0\not\equiv 0\Rightarrow$ & $2$ & $022$ & $(i+1,k-1)$ \\
& $i+k\equiv 1$ and $i\neq m-1$ & $0$ & $i+1\not\equiv 0\Rightarrow$ & $1$ & $-1+0+1=0\Rightarrow$ & $0$ & $010$ & $(i+2,k-3)$ \\
& otherwise & $0$ & $i+1\not\equiv 0\Rightarrow$ & $1$ & $j+0+1\not\equiv 0\Rightarrow$ & $2$ & $012$ & $(i+1,k-2)$ \\
\bottomrule
\end{tabular}}
\end{table}

This proves all return-map formulas.
\end{proof}

\section{Hamiltonicity of the Route~E return maps}\label{app:routeE-cycle}

\subsection{Setup}

The first-return framework of Section~\ref{sec:even} applies uniformly to all three colors.
Appendix~\ref{app:routeE} supplies the closed-form return maps, and this appendix uses those formulas to recover Hamiltonicity.
The main text already proves color~$2$ and isolates the conceptual obstruction/repair story for Case~II, so only colors~$1$ and~$0$ remain here.

For quick navigation, Table~\ref{tab:appendix-cycle-guide} records the proposition/corollary chain for each color and congruence class.
The working coordinates and defect geometries for each color are summarized in Table~\ref{tab:routeE-summary}; later, Proposition~\ref{prop:routeE-splice} records the resulting arithmetic family-block decompositions and splice permutations.

\begin{table}[H]
\centering
\footnotesize
\renewcommand{\arraystretch}{1.15}
\begin{tabularx}{\textwidth}{@{}c c >{\raggedright\arraybackslash}X >{\raggedright\arraybackslash}X@{}}
\toprule
color & congruence class & first-return / lane-map input & concluding cycle statement \\
\midrule
$2$ & all even $m\ge 6$ & Proposition~\ref{prop:R2-main-data} (main text) & Corollary~\ref{cor:R2-main} \\
$1$ & $m\equiv 0,2\pmod 6$ & Proposition~\ref{prop:R1-caseI} together with Proposition~\ref{prop:routeE-splice} & Corollaries~\ref{cor:R1-caseI} and \ref{cor:R1-all} \\
$1$ & $m\equiv 4\pmod 6$ & Proposition~\ref{prop:R1-caseII} together with Proposition~\ref{prop:routeE-splice} & Corollaries~\ref{cor:R1-caseII} and \ref{cor:R1-all} \\
$0$ & $m\equiv 0,2\pmod 6$ & Proposition~\ref{prop:R0-caseI-data} together with Proposition~\ref{prop:routeE-splice} & Corollaries~\ref{cor:R0-caseI} and \ref{cor:R0-all} \\
$0$ & $m\equiv 4\pmod 6$ & Proposition~\ref{prop:R0-caseII-data} together with Proposition~\ref{prop:routeE-splice} & Corollaries~\ref{cor:R0-caseII} and \ref{cor:R0-all} \\
\bottomrule
\end{tabularx}
\caption{Navigation guide for Appendix~\ref{app:routeE-cycle}.}
\label{tab:appendix-cycle-guide}
\end{table}

The next lemma packages the cycle step that remains after the closed-form formulas for the first-return maps have been derived.

\begin{lemma}[splice graph lemma]\label{lem:splice-graph}
Let $X$ be a finite set partitioned into nonempty ordered blocks
\[
X=F_1\sqcup F_2\sqcup \cdots \sqcup F_r,
\qquad
F_j=(x_{j,0},x_{j,1},\dots,x_{j,\ell_j}).
\]
Assume $T:X\to X$ is such that for some permutation $\pi\in S_r$,
\[
T(x_{j,s})=x_{j,s+1}\qquad (0\le s<\ell_j),
\]
and
\[
T(x_{j,\ell_j})=x_{\pi(j),0}\qquad (1\le j\le r).
\]
Then the cycle decomposition of $T$ is obtained from the cycle decomposition of $\pi$: if
\[
(j_1\,j_2\,\dots\,j_q)
\]
is a cycle of $\pi$, then the concatenation
\[
(F_{j_1}\mid F_{j_2}\mid \cdots \mid F_{j_q})
\]
is a single $T$-cycle.
In particular, $T$ is a single $|X|$-cycle if and only if $\pi$ is a single $r$-cycle.
\end{lemma}

\begin{proof}
Inside each block $F_j$, the map $T$ moves one step forward until it reaches the terminal point $x_{j,\ell_j}$, and then it jumps to the initial point of $F_{\pi(j)}$.
Hence on the union of the blocks belonging to a cycle
\[
(j_1\,j_2\,\dots\,j_q)
\]
of $\pi$, the map $T$ is exactly the successor map on the concatenation
\[
(F_{j_1}\mid F_{j_2}\mid \cdots \mid F_{j_q}).
\]
Different cycles of $\pi$ give disjoint $T$-cycles, so the claim follows.
\end{proof}

In the sequel, the relevant blocks are the ordered arithmetic family-blocks obtained from the long bulk arithmetic runs by absorbing the $O(1)$ bridge vertices created by the defect set into adjacent runs.
Once the first-return formulas for $T_c$ and $\rho_c$ are known, the bulk rule determines the successor relation inside each block, while the defect set determines only the finite splice permutation on the block labels.
This resolves the single-cycle question for the induced lane maps $T_c$; Lemma~\ref{lem:counting} is then used separately to upgrade from $T_c$ being a single $m$-cycle to $R_c$ being a single $m^2$-cycle.
When one of the displayed arithmetic runs is empty for the smallest admissible values of $m$, it is omitted from the corresponding block decomposition.

\subsection{Color \texorpdfstring{$2$}{2}}

The color-$2$ return map is proved in full in the main text (Section~\ref{sec:even-r2}).
Proposition~\ref{prop:R2-main-data} derives the first-return data, and Corollary~\ref{cor:R2-main} establishes that $R_2$ is a single $m^2$-cycle for all even $m\ge 6$.
We do not repeat those orbit traces here.

\subsection{Color \texorpdfstring{$1$}{1}}

By Lemma~\ref{lem:routeE-bulkcoords}, the bulk coordinates for color~$1$ are $\Phi_1(i,k)=(u,t)=(i-k,\,k)$, in which the generic branch is $(u,t)\mapsto(u,t+1)$.
By Lemma~\ref{lem:routeE-finite-defect}\textup{(i)}, the defect set lies on the affine lines $u+t=0$ and $u+2t=m-1$ (and, in Case~II, also $u+t=1$), together with finitely many isolated points.
Following the same pattern, it suffices to list the defect components met before first return and sum the lane increments.

We use a first-return analysis on the transversal
\[
L_1=\{(x,0):x\in \Zm\}\subseteq P_0.
\]

Let $\rho_1(x)$ be the first return time of $R_1$ from $(x,0)$ to $L_1$, and define
\[
R_1^{\rho_1(x)}(x,0)=(T_1(x),0).
\]

In the original $(i,k)$-coordinates, $u=i-k$ is the lane coordinate.
A bulk move $(i,k)\mapsto(i+1,k+1)$ preserves $u$.
The two defect branches act as stalls: $(i,k)\mapsto(i+1,k)$ increases $u$ by $1$ (defect line $u+t=0$, i.e.\ $i=0$), and $(i,k)\mapsto(i+2,k)$ increases $u$ by $2$ (defect line $u+2t=m-1$, i.e.\ diagonal $i+k=m-1$).
Starting from $(x,0)$, the net change in $i$ at first return is exactly the total lane increment contributed by the stalls.

\subsubsection{Case I: \texorpdfstring{$m\equiv 0,2\pmod 6$}{m = 0,2 (mod 6)}}

\begin{proposition}\label{prop:R1-caseI}
Assume $m\equiv 0,2\pmod 6$.
Then the first-return map and return times on $L_1$ are
\[
T_1(x)=
\begin{cases}
2,&x=0,\\
x+3,&1\le x\le m-4,\\
1,&x=m-3,\\
0,&x=m-2,\\
3,&x=m-1,
\end{cases}
\qquad
\rho_1(x)=
\begin{cases}
1,&x=0\ \text{or}\ x=m-2,\\
m+2,&1\le x\le m-4,\\
m+3,&x=m-3\ \text{or}\ x=m-1.
\end{cases}
\]
\end{proposition}

\begin{proof}
The closed-form map $R_1$ has three types of moves:
\begin{itemize}[leftmargin=2em]
\item bulk $(i,k)\mapsto(i+1,k+1)$,
\item $+1$ stalls on $i=0$ together with $(1,m-1)$ and $(m-2,1)$,
\item $+2$ stalls on the diagonal $i+k=m-1$ together with $(0,0)$, $(m-2,0)$ and $(m-1,m-1)$.
\end{itemize}

The immediate cases are read directly from the formula:
\[
(0,0)\mapsto(2,0),\qquad (m-2,0)\mapsto(0,0).
\]
Hence
\[
T_1(0)=2,\quad \rho_1(0)=1,\qquad T_1(m-2)=0,\quad \rho_1(m-2)=1.
\]

Now fix $1\le x\le m-4$.
In both the odd and even sub-cases, the orbit from $(x,0)$ encounters exactly two stalls before first return---one diagonal ($+2$) and one vertical ($+1$), in either order---with total lane increment $3$ and return time $m+2$.

\smallskip
\noindent\textbf{Odd $x$:}\enspace
The lane $u=x$ meets the diagonal $i+k=m-1$ at height $(m{-}1{-}x)/2$, producing a $+2$ stall.
The orbit moves to lane $x+2$, then hits the vertical defect at $(0,m{-}x{-}2)$ for a $+1$ stall.
On the resulting lane $x+3$ the remaining defects lie behind the current clock value.

\smallskip
\noindent\textbf{Even $x$:}\enspace
The lane $u=x$ has no diagonal hit; its first defect is $(0,m{-}x)$ (a $+1$ stall, moving to odd lane $x+1$).
That lane then meets the diagonal at height $m-x/2-1$ for a $+2$ stall.
Again the resulting lane $x+3$ has no further defects before wrap.

\smallskip\noindent
In both cases,
\[
T_1(x)=x+3,\qquad \rho_1(x)=m+2.
\]

It remains to treat the two boundary values.

\smallskip
\noindent\textbf{The point $x=m-3$.}
The orbit begins
\[
(m-3,0)\to (m-2,1)\to (m-1,1)\to (0,2)\to (1,2).
\]
Thus it first gets a $+1$ stall at $(m-2,1)$ and then another $+1$ stall at $(0,2)$.
At that point the orbit lies on the lane $u=i-k\equiv m-1$.
The unique diagonal defect on this lane before $k$ wraps is
\[
\left(\frac m2-1,\frac m2\right),
\]
so the orbit continues
\[
(1,2)\xrightarrow{\bulk^{\frac m2-2}}\left(\frac m2-1,\frac m2\right)
\xrightarrow{+2}\left(\frac m2+1,\frac m2\right)
\xrightarrow{\bulk^{\frac m2}}(1,0).
\]
Hence the third stall is this $+2$ diagonal stall, there are exactly three stalls in total, and therefore
\[
T_1(m-3)=1,\qquad \rho_1(m-3)=m+3.
\]

\smallskip
\noindent\textbf{The point $x=m-1$.}
The orbit begins
\[
(m-1,0)\to (0,1)\to (1,1),
\]
so it gets a $+1$ stall at $(0,1)$.
Later it reaches $(m-1,m-1)$, giving a $+2$ stall, and then $(1,m-1)$, giving the final $+1$ stall.
At that point the orbit is on lane $3$ with current clock value $k=m-1$.
The vertical defect on lane $3$ is at height $m-3$, while a diagonal defect on lane $3$ would require
\[
k=\frac{m-4}{2},
\]
so both possible defect heights lie strictly behind the current value $m-1$.
Hence there are no more defects before return.
Therefore the total lane increment is $4$ and the number of stalls is $3$, so
\[
T_1(m-1)=3,\qquad \rho_1(m-1)=m+3.
\]

This proves the stated formulas.
\end{proof}

\subsubsection{Case II: \texorpdfstring{$m\equiv 4\pmod 6$}{m = 4 (mod 6)}}

\begin{proposition}\label{prop:R1-caseII}
Assume $m\equiv 4\pmod 6$.
Then the first-return map and return times on $L_1$ are
\[
T_1(x)=
\begin{cases}
2,&x=0,\\
3,&x=1,\\
5,&x=2,\\
x+2,&x\text{ even},\ x\notin\{0,2,m-2\},\\
0,&x=m-2,\\
x+6,&x\text{ odd},\ x\notin\{1,m-3,m-1\},\\
4,&x=m-3,\\
7,&x=m-1,
\end{cases}
\]
and
\[
\rho_1(x)=
\begin{cases}
1,&x\in\{0,1,m-2\},\\
m+3,&x=2,\\
m+2,&x\text{ even},\ x\notin\{0,2,m-2\},\\
m+4,&x\text{ odd},\ x\notin\{1,m-3,m-1\},\\
m+6,&x\in\{m-3,m-1\}.
\end{cases}
\]
\end{proposition}

\begin{proof}
In Case II the closed-form map has
\begin{itemize}[leftmargin=2em]
\item bulk $(i,k)\mapsto(i+1,k+1)$,
\item $+1$ stalls on the lines $i=0$ and $i=1$, together with the special points
$(2,m-2)$, $(2,m-1)$, $(m-2,1)$,
\item $+2$ stalls on the diagonal $i+k=m-1$ for $2\le i\le m-3$, together with the special points
$(0,0)$, $(1,0)$, $(m-2,0)$, $(m-1,m-1)$.
\end{itemize}

The immediate cases are read directly from the formula:
\[
(0,0)\mapsto(2,0),\qquad (1,0)\mapsto(3,0),\qquad (m-2,0)\mapsto(0,0),
\]
so
\[
T_1(0)=2,\quad T_1(1)=3,\quad T_1(m-2)=0,
\]
and all three return times equal $1$.

For $x=2$, the lane $u=2$ has no diagonal defect.
The orbit runs generically to $(0,m-2)$, then encounters three consecutive $+1$ stalls:
\[
(2,0)\xrightarrow{\bulk^{m-2}}(0,m-2)\xrightarrow{+1}(1,m-2)\xrightarrow{+1}(2,m-2)\xrightarrow{+1}(3,m-2).
\]
From there two bulk moves return to $L_1$.
Hence
\[
T_1(2)=5,\qquad \rho_1(2)=m+3.
\]

Now let $x$ be even with $4\le x\le m-4$.
The lane $u=x$ has no diagonal defect; its first defect is $(0,m-x)$, followed immediately by the $i=1$ defect.
Because $x\neq 2$, these are the only two stalls, giving lane increment $2$ and
\[
T_1(x)=x+2,\qquad \rho_1(x)=m+2.
\]

Next let $x$ be odd with $3\le x\le m-5$.
The orbit encounters four stalls in the pattern $+2,+1,+1,+2$ (one diagonal, two vertical, one diagonal), with total lane increment $6$ and return time $m+4$.
For the endpoint $x=m-5$, the same four-stall pattern applies; after the second diagonal stall the orbit lies on lane $1$ at height $m/2$, and the point $(0,m-1)$ is not a defect (it is excluded from both the $i=0$ and diagonal branches), so the orbit returns to $L_1$ normally.
Hence for all odd $x\in\{3,\ldots,m-5\}$,
\[
T_1(x)=x+6,\qquad \rho_1(x)=m+4.
\]

For $x=m-3$, the orbit is
\[
\begin{aligned}
(m-3,0)&\xrightarrow{\bulk}(m-2,1)\xrightarrow{+1}(m-1,1)\xrightarrow{\bulk}(0,2)\xrightarrow{+1}(1,2)\\
&\xrightarrow{+1}(2,2)\xrightarrow{\bulk^{m-3}}(m-1,m-1)\xrightarrow{+2}(1,m-1)\\
&\xrightarrow{+1}(2,m-1)\xrightarrow{+1}(3,m-1)\xrightarrow{\bulk}(4,0).
\end{aligned}
\]
So the lane increment is $1+1+1+2+1+1=7$ and the return time is $m+6$.
Hence
\[
T_1(m-3)=4,\qquad \rho_1(m-3)=m+6.
\]

For $x=m-1$, the orbit is
\[
\begin{aligned}
(m-1,0)&\xrightarrow{\bulk}(0,1)\xrightarrow{+1}(1,1)\xrightarrow{+1}(2,1)
\xrightarrow{\bulk^{\frac m2-2}}\left(\frac m2,\frac m2-1\right)\\
&\xrightarrow{+2}\xrightarrow{\bulk^{\frac m2-2}}(0,m-3)\xrightarrow{+1}(1,m-3)
\xrightarrow{+1}(2,m-3)\\
&\xrightarrow{+2}(4,m-3)\xrightarrow{\bulk^3}(7,0).
\end{aligned}
\]
Thus the orbit has four $+1$ stalls and two $+2$ stalls, for total lane increment $8$ and return time
$m+6$.
Therefore
\[
T_1(m-1)=7,\qquad \rho_1(m-1)=m+6.
\]
\end{proof}

\subsection{Color \texorpdfstring{$0$}{0}}

By Lemma~\ref{lem:routeE-bulkcoords}, the bulk coordinates for color~$0$ are $\Phi_0(i,k)=(u,t)=(i+2k,\,k)$, in which the generic branch is $(u,t)\mapsto(u,t+1)$.
By Lemma~\ref{lem:routeE-finite-defect}\textup{(iii)}, the defect set lies on the lines $t=0$ and $u=t+1$ (and, in Case~II, also $u=1+2t$), together with finitely many isolated points.
Since the bulk frame already uses the convenient coordinates $(x,y)=(u,t)=(i+2k,\,k)$, we work directly in this frame.

The transversal is
\[
L_0=\{(x,0):x\in \Zm\}.
\]
The generic move freezes $x$ and increments $y$, so $y$ is literally a clock while the exceptional branches change only the lane coordinate $x$.

\subsubsection{Case I: \texorpdfstring{$m\equiv 0,2\pmod 6$}{m = 0,2 (mod 6)}}

\begin{lemma}\label{lem:R0-caseI-xy}
Assume $m\equiv 0,2\pmod 6$.
In the coordinates $(x,y)=(i+2k,k)$, the return map $R_0$ becomes
\[
R_0(x,y)=
\begin{cases}
(m-2,0),&(x,y)=(0,0),\\
(2,2),&(x,y)=(m-1,m-1),\\
(0,2),&(x,y)=(m-2,0),\\
(0,1),&(x,y)=(1,1),\\
(x+1,y+2),&x=y+1,\ (x,y)\neq (1,0),\\
(x+1,y+1),&\bigl(y=0,\ 1\le x\le m-3\bigr)\ \text{or}\ (x,y)\in\{(0,1),(m-2,m-1)\},\\
(x,y+1),&\text{otherwise}.
\end{cases}
\]
All arithmetic is modulo $m$.
\end{lemma}

\begin{proof}
Substitute $i=x-2y$ and $k=y$ into the Case~I formula for $R_0$ in Appendix~\ref{app:routeE} and simplify branch by
branch.
The generic move $(i,k)\mapsto(i-2,k+1)$ becomes $(x,y)\mapsto(x,y+1)$, and the six exceptional
branches transform exactly as displayed.
\end{proof}

Write
\[
G(x,y)=(x,y+1)
\]
for the generic branch in Lemma~\ref{lem:R0-caseI-xy}.
The six exceptional moves will be denoted $P,Q,R,S,A,B$ in the order displayed in the lemma:
$P$ and $Q$ are isolated-point defects at $(0,0)$ and $(m-1,m-1)$;
$R$ maps $(m-2,0)$ to $(0,2)$;
$S$ maps $(1,1)$ to $(0,1)$;
$A$ is the defect on the line $u=t+1$ (i.e.\ $x=y+1$);
$B$ is the defect on the line $t=0$ (i.e.\ $y=0$) together with the boundary points $(0,1)$ and $(m-2,m-1)$.

\begin{proposition}\label{prop:R0-caseI-data}
Assume $m\equiv 0,2\pmod 6$.
Then the first-return map on $L_0$ is
\[
T_0(x)=
\begin{cases}
m-2,&x=0,\\
x+2,&1\le x\le m-5,\\
m-1,&x=m-4,\\
2,&x=m-3,\\
1,&x=m-2,\\
0,&x=m-1,
\end{cases}
\]
and the return times are
\[
\rho_0(x)=
\begin{cases}
1,&x=0,\\
m-1,&1\le x\le m-4\ \text{or}\ x=m-1,\\
2m-3,&x=m-3,\\
2m-1,&x=m-2.
\end{cases}
\]
\end{proposition}

\begin{proof}
Fix $x\in \Zm$ and start from $(x,0)\in L_0$.

\smallskip
\noindent\textbf{Case $x=0$.}
By branch $P$,
\[
(0,0)\xrightarrow{P}(m-2,0),
\]
so $T_0(0)=m-2$ and $\rho_0(0)=1$.

\smallskip
\noindent\textbf{Case $1\le x\le m-5$.}
The orbit is
\[
(x,0)\xrightarrow{B}(x+1,1)\xrightarrow{G^{x-1}}(x+1,x)\xrightarrow{A}(x+2,x+2)
\xrightarrow{G^{m-x-2}}(x+2,0).
\]
Indeed, on the lane $x+1$ the first defect above $y=1$ is the $A$-point $(x+1,x)$.
After that jump the orbit lies on lane $x+2$ at height $y=x+2$.
On this lane the only possible $A$-point is $(x+2,x+1)$, whose height $x+1$ is already behind the current
clock value.
The special points $P,Q,R,S$ lie on different lanes, and the $B$-branch occurs only at $y=0$ or at the
isolated points $(0,1)$ and $(m-2,m-1)$, none of which lies ahead on lane $x+2$.
Thus every remaining step before $y=0$ is generic.
Therefore
\[
T_0(x)=x+2,\qquad
\rho_0(x)=1+(x-1)+1+(m-x-2)=m-1.
\]

\smallskip
\noindent\textbf{Case $x=m-4$.}
Here
\[
(m-4,0)\xrightarrow{B}(m-3,1)\xrightarrow{G^{m-5}}(m-3,m-4)\xrightarrow{A}(m-2,m-2)
\xrightarrow{G}(m-2,m-1)\xrightarrow{B}(m-1,0).
\]
So
\[
T_0(m-4)=m-1,\qquad
\rho_0(m-4)=1+(m-5)+1+1+1=m-1.
\]

\smallskip
\noindent\textbf{Case $x=m-3$.}
The orbit is
\[
(m-3,0)\xrightarrow{B}(m-2,1)\xrightarrow{G^{m-4}}(m-2,m-3)\xrightarrow{A}(m-1,m-1)
\xrightarrow{Q}(2,2)\xrightarrow{G^{m-2}}(2,0).
\]
Hence
\[
T_0(m-3)=2,\qquad
\rho_0(m-3)=1+(m-4)+1+1+(m-2)=2m-3.
\]

\smallskip
\noindent\textbf{Case $x=m-2$.}
The orbit is
\[
(m-2,0)\xrightarrow{R}(0,2)\xrightarrow{G^{m-3}}(0,m-1)\xrightarrow{A}(1,1)\xrightarrow{S}(0,1)
\xrightarrow{B}(1,2)\xrightarrow{G^{m-2}}(1,0).
\]
Therefore
\[
T_0(m-2)=1,\qquad
\rho_0(m-2)=1+(m-3)+1+1+1+(m-2)=2m-1.
\]

\smallskip
\noindent\textbf{Case $x=m-1$.}
The orbit is
\[
(m-1,0)\xrightarrow{G^{m-2}}(m-1,m-2)\xrightarrow{A}(0,0).
\]
So
\[
T_0(m-1)=0,\qquad \rho_0(m-1)=m-2+1=m-1.
\]

These cases exhaust $x\in \Zm$.
\end{proof}

\subsubsection{Case II: \texorpdfstring{$m\equiv 4\pmod 6$}{m = 4 (mod 6)}}

\begin{lemma}\label{lem:R0-caseII-xy}
Assume $m\equiv 4\pmod 6$.
In the coordinates $(x,y)=(i+2k,k)$, the return map $R_0$ becomes
\[
R_0(x,y)=
\begin{cases}
(m-2,0),&(x,y)=(0,0),\\
(x+3,y+3),&(x,y)\in\{(m-1,m-1),(m-2,m-2)\},\\
(x+2,y+2),&\substack{(x,y)\in\{(m\!-\!2,0),(0,m\!-\!1)\}\ \text{or}\ x\equiv 1+2y\pmod m,\\ 0\le y\le m-3,}\\
(0,1),&(x,y)=(1,1),\\
(x+1,y+2),&x=y+1,\ (x,y)\notin\{(1,0),(0,m-1)\},\\
(x+1,y+1),&\substack{y=0,\ 2\le x\le m-3,\\ \text{or }(x,y)\in\{(0,1),(m-3,m-2),(m-2,m-1)\},}\\
(x,y+1),&\text{otherwise}.
\end{cases}
\]
All arithmetic is modulo $m$.
\end{lemma}

\begin{proof}
Again substitute $i=x-2y$ and $k=y$ into the Case~II formula from Appendix~\ref{app:routeE}.
The generic branch becomes $(x,y)\mapsto(x,y+1)$, and the exceptional branches simplify exactly as shown.
\end{proof}

As before write
\[
G(x,y)=(x,y+1)
\]
for the generic move.
For an odd lane $u\in \Zm$ define the two $R$-heights
\[
r_-(u)=\frac{u-1}{2},
\qquad
r_+(u)=\frac{u-1}{2}+\frac m2.
\]
These are exactly the two solutions $y\in \{0,\dots,m-1\}$ of
\[
u\equiv 1+2y\pmod m.
\]

\paragraph{Reading note for Case~II color $0$.}
Relative to Case~I, the only new long defect component is the extra $R$-branch on the affine line $x\equiv 1+2y\pmod m$.
It is therefore helpful to sort the lanes before reading Proposition~\ref{prop:R0-caseII-data}:
\begin{center}
{\footnotesize
\renewcommand{\arraystretch}{1.15}
\begin{tabularx}{\textwidth}{@{}ll>{\raggedright\arraybackslash}X@{}}
\toprule
family & lanes & role in the proof \\
\midrule
special point & $x=0$ & immediate bridge to $m-2$ in one step \\
warm-up lanes & $x=1,2$ & first encounters with the Case~II repair family \\
generic lanes & $3\le x\le m-7$ & the parity only changes the order of the defect encounters; in both parities the net lane increment is $+4$ and the return time is $m-2$ \\
boundary splice lanes & $x\in\{m-6,m-5,m-4,m-3,m-2,m-1\}$ & the only lanes where the isolated points $Q,S$ or the top $R$-hit alter the generic pattern and determine the final residue-$4$ splice \\
\bottomrule
\end{tabularx}
}
\end{center}
So the proposition is easiest to read in exactly that order: first the special and warm-up lanes, then the generic family, and finally the six boundary splice lanes.

\begin{proposition}\label{prop:R0-caseII-data}
Assume $m\equiv 4\pmod 6$.
Then the first-return map on $L_0$ is given by
\[
T_0(0)=m-2,\qquad T_0(1)=4,\qquad T_0(2)=6,\qquad T_0(x)=x+4\ \ (3\le x\le m-7),
\]
together with
\[
T_0(m-6)=3,\quad T_0(m-5)=m-1,\quad T_0(m-4)=0,
\]
\[
T_0(m-3)=2,\quad T_0(m-2)=5,\quad T_0(m-1)=1,
\]
and the return times are
\[
\rho_0(0)=1,\qquad \rho_0(x)=m-2 \quad (\text{for } 1\le x\le m-7 \text{ or } x=m-4),
\]
\[
\rho_0(m-6)=2m-4,\quad \rho_0(m-2)=2m-4,\quad \rho_0(m-5)=m-1,
\]
\[
\rho_0(m-1)=m-1,\quad
\rho_0(m-3)=2m-3.
\]
\end{proposition}

\begin{proof}
Fix $x\in \Zm$ and start from $(x,0)\in L_0$.

\smallskip
\noindent\textbf{Case $x=0$.}
By the special point branch,
\[
(0,0)\xrightarrow{P}(m-2,0),
\]
so $T_0(0)=m-2$ and $\rho_0(0)=1$.

\smallskip
\noindent\textbf{Case $x=1$.}
The orbit is
\[
(1,0)\xrightarrow{R}(3,2)\xrightarrow{A}(4,4)\xrightarrow{G^{m-4}}(4,0).
\]
So
\[
T_0(1)=4,\qquad \rho_0(1)=1+1+(m-4)=m-2.
\]

\smallskip
\noindent\textbf{Case $x=2$.}
The orbit is
\[
(2,0)\xrightarrow{B}(3,1)\xrightarrow{R}(5,3)\xrightarrow{G}(5,4)\xrightarrow{A}(6,6)
\xrightarrow{G^{m-6}}(6,0),
\]
hence
\[
T_0(2)=6,\qquad \rho_0(2)=1+1+1+1+(m-6)=m-2.
\]

\smallskip
\noindent\textbf{Case $3\le x\le m-7$ (generic lanes).}
In both the odd and even sub-cases, the orbit starting from $(x,0)$ encounters exactly one $B$-stall (at $y=0$), one or two traversals of the $R$-branch (depending on parity), and one $A$-stall, with total lane increment $4$ and return time $m-2$.

\smallskip
\noindent\textbf{Odd $x$:}\enspace
The orbit follows $B\to G^{x-1}\to A\to G^{\frac{m-x-3}{2}}\to R\to G^{\frac{m-x-5}{2}}$, encountering the $A$-point at height $x$ and the upper $R$-point at height $r_+(x+2)$.

\smallskip
\noindent\textbf{Even $x$:}\enspace
The orbit follows $B\to G^{x/2-1}\to R\to G^{x/2}\to A\to G^{m-x-4}$, encountering the lower $R$-point at height $x/2$ and the $A$-point at height $x+2$.

\smallskip\noindent
In both cases,
\[
T_0(x)=x+4,\qquad \rho_0(x)=m-2.
\]

\smallskip
\noindent\textbf{Case $x=m-6$.}
Here the orbit is
\begin{multline*}
(m-6,0)\xrightarrow{B}(m-5,1)
\xrightarrow{G^{\frac{m-8}{2}}}\left(m-5,\frac{m-6}{2}\right)\xrightarrow{R}\left(m-3,\frac m2-1\right) \\
\xrightarrow{G^{\frac{m-6}{2}}}(m-3,m-4)\xrightarrow{A}(m-2,m-2)
\end{multline*}
\[
\xrightarrow{Q}(1,1)\xrightarrow{S}(0,1)\xrightarrow{B}(1,2)
\xrightarrow{G^{\frac{m-4}{2}}}\left(1,\frac m2\right)\xrightarrow{R}\left(3,\frac m2+2\right)
\xrightarrow{G^{\frac{m-4}{2}}}(3,0).
\]
Thus
\[
T_0(m-6)=3,
\]
and
\[
\rho_0(m-6)=
1+\frac{m-8}{2}+1+\frac{m-6}{2}+1+1+1+1+\frac{m-4}{2}+1+\frac{m-4}{2}
=
2m-4.
\]

\smallskip
\noindent\textbf{Case $x=m-5$.}
The orbit is
\begin{multline*}
(m-5,0)\xrightarrow{B}(m-4,1)\xrightarrow{G^{m-6}}(m-4,m-5)\xrightarrow{A}(m-3,m-3) \\
\xrightarrow{G}(m-3,m-2)\xrightarrow{B}(m-2,m-1)\xrightarrow{B}(m-1,0).
\end{multline*}
Therefore
\[
T_0(m-5)=m-1,\qquad
\rho_0(m-5)=1+(m-6)+1+1+1+1=m-1.
\]

\smallskip
\noindent\textbf{Case $x=m-4$.}
The orbit is
\[
(m-4,0)\xrightarrow{B}(m-3,1)\xrightarrow{G^{\frac m2-3}}\left(m-3,\frac{m-4}{2}\right)
\xrightarrow{R}\left(m-1,\frac m2\right)\xrightarrow{G^{\frac m2-2}}(m-1,m-2)\xrightarrow{A}(0,0).
\]
So
\[
T_0(m-4)=0,\qquad
\rho_0(m-4)=1+\left(\frac m2-3\right)+1+\left(\frac m2-2\right)+1=m-2.
\]

\smallskip
\noindent\textbf{Case $x=m-3$.}
The orbit is
\[
(m-3,0)\xrightarrow{B}(m-2,1)\xrightarrow{G^{m-4}}(m-2,m-3)\xrightarrow{A}(m-1,m-1)
\xrightarrow{Q}(2,2)\xrightarrow{G^{m-2}}(2,0).
\]
Hence
\[
T_0(m-3)=2,\qquad
\rho_0(m-3)=1+(m-4)+1+1+(m-2)=2m-3.
\]

\smallskip
\noindent\textbf{Case $x=m-2$.}
The orbit is
\begin{multline*}
(m-2,0)\xrightarrow{R}(0,2)\xrightarrow{G^{m-3}}(0,m-1)\xrightarrow{R}(2,1)\xrightarrow{A}(3,3) \\
\xrightarrow{G^{\frac{m-4}{2}}}\left(3,\frac m2+1\right)\xrightarrow{R}\left(5,\frac m2+3\right)
\xrightarrow{G^{\frac{m-6}{2}}}(5,0).
\end{multline*}
Therefore
\[
T_0(m-2)=5,
\]
and
\[
\rho_0(m-2)=1+(m-3)+1+1+\frac{m-4}{2}+1+\frac{m-6}{2}=2m-4.
\]

\smallskip
\noindent\textbf{Case $x=m-1$.}
The orbit is
\[
(m-1,0)\xrightarrow{G^{\frac m2-1}}\left(m-1,\frac m2-1\right)\xrightarrow{R}\left(1,\frac m2+1\right)
\xrightarrow{G^{\frac m2-1}}(1,0).
\]
Hence
\[
T_0(m-1)=1,\qquad
\rho_0(m-1)=\left(\frac m2-1\right)+1+\left(\frac m2-1\right)=m-1.
\]

These cases exhaust $x\in \Zm$.
\end{proof}

\subsection{Splice graphs and completion}

The preceding propositions determine the explicit first-return maps $T_c$ and return times $\rho_c$.
We now compress the single-cycle step for the induced lane maps into a finite splice-permutation statement on arithmetic family-blocks.

\begin{proposition}[defect-set splice normal form for the Route~E lane maps]\label{prop:routeE-splice}
After the first-return formulas are known, the defect set does not create any new long-range dynamics inside the bulk arithmetic runs.
Instead, it only determines where finitely many ordered arithmetic family-blocks are cut and how they are spliced together.
After absorbing the isolated bridge vertices into adjacent arithmetic runs, the block data are as follows.

\medskip
\noindent\textbf{Color $2$.}
\begin{center}
{\footnotesize
\renewcommand{\arraystretch}{1.15}
\begin{tabularx}{\textwidth}{@{}l>{\raggedright\arraybackslash}Xc@{}}
\toprule
case & ordered family-blocks & splice \\
\midrule
all even $m\ge 6$ & $F_{2,1}=(0,1)$; $F_{2,2}=(m-1,m-2,\dots,2)$ & $(1\,2)$ \\
\bottomrule
\end{tabularx}
}
\end{center}

\medskip
\noindent\textbf{Color $1$.}
\begin{center}
{\footnotesize
\renewcommand{\arraystretch}{1.15}
\begin{tabularx}{\textwidth}{@{}l>{\raggedright\arraybackslash}Xc@{}}
\toprule
case & ordered family-blocks & splice \\
\midrule
$m\equiv 2\pmod 6$ & $F_{1,1}=(0,2,5,\dots,m-3)$; $F_{1,2}=(1,4,7,\dots,m-1)$; $F_{1,3}=(3,6,9,\dots,m-2)$ & $(1\,2\,3)$ \\
$m\equiv 0\pmod 6$ & $F_{1,1}=(0,2,5,\dots,m-1)$; $F_{1,2}=(3,6,9,\dots,m-3,1)$; $F_{1,3}=(4,7,10,\dots,m-2)$ & $(1\,2\,3)$ \\
$m\equiv 4\pmod 6$ & $F_{1,1}=(0,2,5,11,\dots,m-5,1)$; $F_{1,2}=(3,9,15,\dots,m-1,7,13,\dots,m-3)$; $F_{1,3}=(4,6,8,\dots,m-2)$ & $(1\,2\,3)$ \\
\bottomrule
\end{tabularx}
}
\end{center}

\medskip
\noindent\textbf{Color $0$.}
\begin{center}
{\footnotesize
\renewcommand{\arraystretch}{1.15}
\begin{tabularx}{\textwidth}{@{}l>{\raggedright\arraybackslash}Xc@{}}
\toprule
case & ordered family-blocks & splice \\
\midrule
$m\equiv 0,2\pmod 6$ & $F_{0,1}=(0,m-2)$; $F_{0,2}=(1,3,5,\dots,m-3)$; $F_{0,3}=(2,4,6,\dots,m-4)$; $F_{0,4}=(m-1)$ & $(1\,2\,3\,4)$ \\
$m\equiv 10\pmod{12}$ & $F_{0,1}=(0,m-2,5,9,\dots,m-1,1)$; $F_{0,2}=(4,8,\dots,m-6)$; $F_{0,3}=(3,7,\dots,m-3)$; $F_{0,4}=(2,6,\dots,m-4)$ & $(1\,2\,3\,4)$ \\
$m\equiv 4\pmod{12}$ & $F_{0,1}=(0,m-2,5,9,\dots,m-3)$; $F_{0,2}=(2,6,\dots,m-6)$; $F_{0,3}=(3,7,\dots,m-1,1)$; $F_{0,4}=(4,8,\dots,m-4)$ & $(1\,2\,3\,4)$ \\
\bottomrule
\end{tabularx}
}
\end{center}

In every row, the map $T_c$ advances to the next term inside each block, and the terminal point of $F_{c,j}$ maps to the initial point of $F_{c,\pi_c(j)}$.
Consequently each $T_c$ is a single $m$-cycle.
\end{proposition}

\begin{proof}
For $T_2$, Proposition~\ref{prop:R2-main-data} gives
\[
T_2(0)=1,\qquad T_2(1)=m-1,\qquad T_2(2)=0,\qquad T_2(x)=x-1\quad (3\le x\le m-1),
\]
which is exactly the successor rule inside the blocks $F_{2,1},F_{2,2}$ with splice permutation $\pi_2=(1\,2)$.

For color~$1$, Proposition~\ref{prop:R1-caseI} yields the Case~I formulas and Proposition~\ref{prop:R1-caseII} the Case~II formulas.
In Case~I the generic rule $x\mapsto x+3$ produces the three long arithmetic families, while the special values $0,m-3,m-2,m-1$ give the terminal-to-initial jumps recorded in the displayed blocks.
In Case~II the even rule $x\mapsto x+2$ and odd rule $x\mapsto x+6$ produce the three displayed blocks, and the special values $0,1,2,m-3,m-2,m-1$ again determine only the block-to-block jumps.
Thus in every congruence class the induced splice permutation is $\pi_1=(1\,2\,3)$.

For color~$0$, Proposition~\ref{prop:R0-caseI-data} gives the Case~I formulas with generic rule $x\mapsto x+2$, and Proposition~\ref{prop:R0-caseII-data} gives the Case~II formulas with generic rule $x\mapsto x+4$.
In Case~I the exceptional values $0,m-4,m-3,m-2,m-1$ cut the odd and even arithmetic families into the four displayed blocks.
For Case~II, we spell out the nontrivial closure because it is the most bookkeeping-heavy row.
The generic rule is $x\mapsto x+4$ on $3\le x\le m-7$.
If $m\equiv 10\pmod{12}$, Proposition~\ref{prop:R0-caseII-data} gives the special values
\[
T_0(0)=m-2,\quad T_0(m-2)=5,\quad T_0(m-1)=1,\quad T_0(1)=4,
\]
\[
T_0(m-6)=3,\quad T_0(m-3)=2,\quad T_0(m-4)=0.
\]
Hence the orbit segments are
\[
0\to m-2\to 5\to 9\to\cdots\to m-1\to 1,
\]
\[
4\to 8\to\cdots\to m-6,\qquad 3\to 7\to\cdots\to m-3,\qquad 2\to 6\to\cdots\to m-4,
\]
which are exactly the displayed blocks $F_{0,1},F_{0,2},F_{0,3},F_{0,4}$.
Their terminal images are
\[
1\mapsto 4,\qquad m-6\mapsto 3,\qquad m-3\mapsto 2,\qquad m-4\mapsto 0,
\]
so the terminals of $F_{0,1},F_{0,2},F_{0,3},F_{0,4}$ map to the initials of $F_{0,2},F_{0,3},F_{0,4},F_{0,1}$, giving the splice permutation $(1\,2\,3\,4)$.
If instead $m\equiv 4\pmod{12}$, the same generic rule together with
\[
T_0(0)=m-2,\quad T_0(m-2)=5,\quad T_0(m-3)=2,\quad T_0(m-6)=3,\quad T_0(1)=4,\quad T_0(m-4)=0
\]
yields the blocks
\[
0\to m-2\to 5\to 9\to\cdots\to m-3,
\]
\[
2\to 6\to\cdots\to m-6,\qquad 3\to 7\to\cdots\to m-1\to 1,\qquad 4\to 8\to\cdots\to m-4,
\]
with terminal images
\[
m-3\mapsto 2,\qquad m-6\mapsto 3,\qquad 1\mapsto 4,\qquad m-4\mapsto 0.
\]
Again the terminals map to the initials of the next displayed blocks, so the splice permutation is $(1\,2\,3\,4)$.
Thus in every congruence class the induced splice permutation is $\pi_0=(1\,2\,3\,4)$.

Therefore the hypotheses of Lemma~\ref{lem:splice-graph} hold in each row.
Since every displayed $\pi_c$ is a single cycle, each induced lane map $T_c$ is a single $m$-cycle.
\end{proof}

For color~$2$, Proposition~\ref{prop:routeE-splice} recovers the cyclic order already used in Corollary~\ref{cor:R2-main}.
For colors~$1$ and $0$, the remaining corollaries combine Proposition~\ref{prop:routeE-splice} with the return-time sums from the first-return formulas.

\begin{corollary}\label{cor:R1-caseI}
If $m\equiv 0,2\pmod 6$, then $R_1$ is a single $m^2$-cycle.
\end{corollary}

\begin{proof}
In Case~I, Proposition~\ref{prop:routeE-splice} shows that $T_1$ is a single $m$-cycle.
Also, Proposition~\ref{prop:R1-caseI} gives
\[
\sum_{x\in \Zm}\rho_1(x)=\underbrace{2}_{x\in\{0,m{-}2\}}+\underbrace{(m-4)(m+2)}_{\text{generic}}+\underbrace{2(m+3)}_{\text{boundary}}=m^2.
\]
Now apply Lemma~\ref{lem:counting}.
\end{proof}

\begin{corollary}\label{cor:R1-caseII}
If $m\equiv 4\pmod 6$, then $R_1$ is a single $m^2$-cycle.
\end{corollary}

\begin{proof}
In Case~II, Proposition~\ref{prop:routeE-splice} shows that $T_1$ is a single $m$-cycle.
Also, Proposition~\ref{prop:R1-caseII} gives
\[
\sum_{x\in \Zm}\rho_1(x)
=
\underbrace{3}_{\text{imm.}}+\underbrace{(m+3)}_{x=2}+\underbrace{\tfrac{m-6}{2}(m+2)}_{\text{gen.\ even}}+\underbrace{\tfrac{m-6}{2}(m+4)}_{\text{gen.\ odd}}+\underbrace{2(m+6)}_{\text{bdy.}}=m^2.
\]
Therefore Lemma~\ref{lem:counting} applies.
\end{proof}

\begin{corollary}\label{cor:R1-all}
For every even $m\ge 6$, the return map $R_1$ is a single $m^2$-cycle on $P_0$.
\end{corollary}

\begin{proof}
Combine Corollaries~\ref{cor:R1-caseI} and \ref{cor:R1-caseII}.
\end{proof}

\begin{corollary}\label{cor:R0-caseI}
If $m\equiv 0,2\pmod 6$, then $R_0$ is a single $m^2$-cycle.
\end{corollary}

\begin{proof}
In Case~I, Proposition~\ref{prop:routeE-splice} shows that $T_0$ is a single $m$-cycle.
Also Proposition~\ref{prop:R0-caseI-data} gives
\[
\sum_{x\in \Zm}\rho_0(x)=\underbrace{1}_{x=0}+\underbrace{(m-3)(m-1)}_{\text{generic}}+\underbrace{(2m-3)}_{x=m{-}3}+\underbrace{(2m-1)}_{x=m{-}2}=m^2.
\]
Lemma~\ref{lem:counting} now applies.
\end{proof}

\begin{corollary}\label{cor:R0-caseII}
If $m\equiv 4\pmod 6$, then $R_0$ is a single $m^2$-cycle.
\end{corollary}

\begin{proof}
In Case~II, Proposition~\ref{prop:routeE-splice} shows that $T_0$ is a single $m$-cycle.
Also Proposition~\ref{prop:R0-caseII-data} gives
\[
\sum_{x\in \Zm}\rho_0(x)=\underbrace{1}_{x=0}+\underbrace{(m-6)(m-2)}_{\text{generic}}+\underbrace{2(2m-4)}_{x\in\{m{-}6,m{-}2\}}+\underbrace{2(m-1)}_{x\in\{m{-}5,m{-}1\}}+\underbrace{(2m-3)}_{x=m{-}3}=m^2.
\]
Lemma~\ref{lem:counting} finishes the proof.
\end{proof}

\begin{corollary}\label{cor:R0-all}
For every even $m\ge 6$, the return map $R_0$ is a single $m^2$-cycle on $P_0$.
\end{corollary}

\begin{proof}
Combine Corollaries~\ref{cor:R0-caseI} and \ref{cor:R0-caseII}.
\end{proof}

\section{An explicit decomposition for \texorpdfstring{$m=4$}{m=4}}\label{app:m4}

For $m=4$, define a direction assignment by specifying at each vertex $(i,j,k)\in(\mathbb Z_4)^3$ a direction triple $(d_0,d_1,d_2)\in\{0,1,2\}^3$, where $0,1,2$ correspond respectively to the directions $+i,+j,+k$.
The entry in the table for fixed $i$ and coordinates $(j,k)$ is the three-digit word $d_0d_1d_2$.

\begin{center}
\begin{minipage}[t]{0.48\linewidth}\centering\small
$i=0$\\[2pt]
\begin{tabular}{c|cccc}
\diagbox{j}{k} & 0 & 1 & 2 & 3\\\hline
0 & \texttt{210} & \texttt{012} & \texttt{120} & \texttt{021}\\
1 & \texttt{201} & \texttt{021} & \texttt{120} & \texttt{210}\\
2 & \texttt{120} & \texttt{012} & \texttt{201} & \texttt{210}\\
3 & \texttt{201} & \texttt{201} & \texttt{210} & \texttt{102}\\
\end{tabular}
\end{minipage}
\hfill
\begin{minipage}[t]{0.48\linewidth}\centering\small
$i=1$\\[2pt]
\begin{tabular}{c|cccc}
\diagbox{j}{k} & 0 & 1 & 2 & 3\\\hline
0 & \texttt{120} & \texttt{210} & \texttt{120} & \texttt{210}\\
1 & \texttt{102} & \texttt{021} & \texttt{201} & \texttt{012}\\
2 & \texttt{021} & \texttt{201} & \texttt{210} & \texttt{120}\\
3 & \texttt{210} & \texttt{201} & \texttt{012} & \texttt{201}\\
\end{tabular}
\end{minipage}

\vspace{1em}
\begin{minipage}[t]{0.48\linewidth}\centering\small
$i=2$\\[2pt]
\begin{tabular}{c|cccc}
\diagbox{j}{k} & 0 & 1 & 2 & 3\\\hline
0 & \texttt{021} & \texttt{210} & \texttt{201} & \texttt{021}\\
1 & \texttt{012} & \texttt{201} & \texttt{120} & \texttt{210}\\
2 & \texttt{210} & \texttt{120} & \texttt{210} & \texttt{102}\\
3 & \texttt{102} & \texttt{102} & \texttt{012} & \texttt{210}\\
\end{tabular}
\end{minipage}
\hfill
\begin{minipage}[t]{0.48\linewidth}\centering\small
$i=3$\\[2pt]
\begin{tabular}{c|cccc}
\diagbox{j}{k} & 0 & 1 & 2 & 3\\\hline
0 & \texttt{021} & \texttt{201} & \texttt{012} & \texttt{120}\\
1 & \texttt{210} & \texttt{210} & \texttt{120} & \texttt{021}\\
2 & \texttt{201} & \texttt{021} & \texttt{201} & \texttt{210}\\
3 & \texttt{201} & \texttt{120} & \texttt{201} & \texttt{210}\\
\end{tabular}
\end{minipage}
\end{center}

The direction table above is the finite witness.
For print readability, the full $64$-term color cycles are not reproduced here; they are stored in the supplementary files \texttt{m4\_witness.json} and \texttt{m4\_cycle\_lists.txt}.
Those files record the three color orbits obtained by iterating from $(0,0,0)$, and the script \texttt{verify\_m4\_witness.py} automates the finite audit:
\begin{enumerate}[leftmargin=2.2em]
\item it checks that every table entry is a permutation of $(0,1,2)$;
\item it follows each color from $(0,0,0)$ using the table and verifies that the orbit has length $64$ before returning;
\item it confirms that the stored machine-readable data and the human-readable cycle lists agree with those computed directly from the table.
\end{enumerate}
Because $D_3(4)$ has only $64$ vertices, this finite verification is completely transparent: the printed table is the human-readable witness, while the supplementary files preserve the full archival orbit data.

\section{Verification artifacts}\label{app:verification}

The verification artifacts focus on the even case $m\ge 6$, which involves the longest case analysis.
The odd construction is given by a short closed form in the main text.
For the boundary case $m=4$, the same directory includes the finite witness \texttt{m4\_witness.json} and its checker \texttt{verify\_m4\_witness.py}.

The main script \texttt{route\_e\_even.py} checks that, for every even $m$ with $6\le m\le 60$:
\begin{enumerate}[leftmargin=2.2em]
\item the resulting direction assignment is a valid coloring (each color is a permutation and the three outgoing arcs at every vertex are distinct);
\item the return maps on $P_0$ agree with the displayed piecewise formulas;
\item the induced return maps on $P_0$ are single $m^2$-cycles;
\item each color permutation is a single Hamilton cycle.
\end{enumerate}
Two auxiliary scripts target the most verification-heavy steps in the appendices:
\begin{enumerate}[leftmargin=2.2em,label=(\alph*)]
\item \texttt{routee\_return\_formula\_tables\_check.py} checks the partition sets, low-layer words, and return-map formulas of Proposition~\ref{prop:routeE-return-formulas};
\item \texttt{routee\_first\_return\_check.py} checks the appendix first-return formulas against the actual return-map dynamics on the stated transversals.
\end{enumerate}
For larger even $m$, the main script can also validate the construction by checking only the closed-form first-return maps and return-time sums.
These computations are not part of the proofs in the present paper; they serve as independent verification of the even-case construction and as regression checks for the displayed closed-form formulas.


\begin{thebibliography}{99}

\bibitem{Kotzig1973}
A.~Kotzig,
\newblock Every Cartesian product of two circuits is decomposable into two Hamiltonian circuits,
\newblock Centre de Recherches Math\'ematiques, Montr\'eal, Rapport 233 (1973).

\bibitem{Foregger1978}
M.~F.~Foregger,
\newblock Hamiltonian decompositions of products of cycles,
\newblock {\em Discrete Mathematics} 24 (1978), 251--260.

\bibitem{Bermond1978}
J.-C.~Bermond,
\newblock Hamilton decomposition of graphs, directed graphs and hypergraphs,
\newblock in {\em Advances in Graph Theory}, Annals of Discrete Mathematics 3 (1978), 21--28.

\bibitem{AlspachBermondSotteau1990}
B.~Alspach, J.-C.~Bermond, and D.~Sotteau,
\newblock Decomposition into cycles I: Hamilton decompositions,
\newblock in {\em Cycles and Rays} (G.~Hahn, G.~Sabidussi, R.~E.~Woodrow, eds.),
NATO ASI Series, vol.~301, Springer, 1990, pp.~9--18.

\bibitem{Stong1991}
R.~Stong,
\newblock Hamilton decompositions of Cartesian products of graphs,
\newblock {\em Discrete Mathematics} 90 (1991), 169--190.

\bibitem{WitteGallian1984}
D.~Witte and J.~A.~Gallian,
\newblock A survey: Hamiltonian cycles in Cayley graphs,
\newblock {\em Discrete Mathematics} 51 (1984), 293--304.

\bibitem{CurranWitte1985}
S.~J.~Curran and D.~Witte,
\newblock Hamilton paths in Cartesian products of directed cycles,
\newblock in {\em Cycles in Graphs}, Annals of Discrete Mathematics 27 (1985), 35--74.

\bibitem{TrotterErdos1978}
W.~T.~Trotter and P.~Erd\H{o}s,
\newblock When the Cartesian product of directed cycles is Hamiltonian,
\newblock {\em Journal of Graph Theory} 2 (1978), 137--142.

\bibitem{CurranGallian1996}
S.~J.~Curran and J.~A.~Gallian,
\newblock Hamiltonian cycles and paths in Cayley graphs and digraphs --- a survey,
\newblock {\em Discrete Mathematics} 156 (1996), 1--18.

\bibitem{Keating1985}
K.~Keating,
\newblock Multiple-ply Hamiltonian graphs and digraphs,
\newblock in {\em Cycles in Graphs}, {\em Annals of Discrete Mathematics} 27 (1985), 81--88.

\bibitem{Bogdanowicz2017}
Z.~R.~Bogdanowicz,
\newblock On decomposition of the Cartesian product of directed cycles into cycles of equal lengths,
\newblock {\em Discrete Applied Mathematics} 229 (2017), 148--150.

\bibitem{Bogdanowicz2020}
Z.~R.~Bogdanowicz,
\newblock Identifying Hamilton cycles in the Cartesian product of directed cycles,
\newblock {\em AKCE International Journal of Graphs and Combinatorics} 17 (2020), no.~1, 534--538.

\bibitem{DarijaniMiraftabMorris2022}
I.~Darijani, B.~Miraftab, and D.~Witte Morris,
\newblock Arc-disjoint Hamiltonian paths in Cartesian products of directed cycles,
\newblock {\em Ars Mathematica Contemporanea} 25 (2025), no.~2, P2.10.

\bibitem{Mohar2006}
B.~Mohar,
\newblock Kempe equivalence of colorings,
\newblock in {\em Graph Theory in Paris}, Trends in Mathematics, Birkh\"auser, 2006, pp.~287--297.

\bibitem{StiebitzEtAl2012}
M.~Stiebitz, D.~Scheide, B.~Toft, and L.~M.~Favrholdt,
\newblock {\em Graph Edge Coloring: Vizing's Theorem and Goldberg's Conjecture},
\newblock Wiley Series in Discrete Mathematics and Optimization, Wiley, 2012.

\bibitem{Knuth2026}
D.~E.~Knuth,
\newblock Claude's cycles,
\newblock Preprint, March 2026.
\newblock \url{https://www-cs-faculty.stanford.edu/~knuth/papers/claude-cycles.pdf}.

\bibitem{AquinoMichaels2026}
K.~Aquino-Michaels,
\newblock Completing Claude's cycles,
\newblock Preprint, March 2026.
\newblock \url{https://github.com/no-way-labs/residue}.

\end{thebibliography}
\end{document}